%% file: ex_article.tex
\documentclass[review,onefignum,onetabnum]{siamart171218}

\input{ex_shared}

\ifpdf
\hypersetup{
	pdftitle={Double Cone Structure in Subwavelength Regime},
	pdfauthor={Borui Miao, Yi Zhu}
}
\fi

\usepackage{amsmath,amssymb}
\begin{document}
	
	\maketitle
	\begin{abstract}
		Honeycomb structures lead to conically degenerate points on the dispersion surfaces. These spectral points, termed as Dirac points, are responsible for various topological phenomena. In this paper, we investigate the generalized honeycomb-structured materials, which have six inclusions in a hexagonal cell. We obtain the asymptotic band structures and corresponding eigenstates in the subwavelength regime using the layer potential theory. Specifically, we rigorously prove the existence of the double Dirac cones lying on the 2nd-5th bands when the six inclusions satisfy an additional symmetry. This type of inclusions will be referred to as super honeycomb-structured inclusions. Two distinct deformations breaking the additional symmetry, contraction and dilation, are further discussed. We prove that the double Dirac cone disappears, and a local spectral gap opens. The corresponding eigenstates are also obtained to show the topological differences between these two deformations. Direct numerical simulations using finite element methods agree well with our analysis.
	\end{abstract}
	\begin{keywords}
		subwavelength regime, generalized honeycomb-structured inclusions, layer potentials, periodic capacitance matrix.
	\end{keywords}
	
	\begin{AMS}
		35C15, 35C20, 35P15, 35B40, 45M05
	\end{AMS}
	
	\section{Introduction}
	The last two decades have witnessed a vast development in topological materials admitting various intriguing phenomena \cite{Hasan2010, Lan2022, Ozawa2019, Qi2011}, which are rooted in the symmetries of the materials. A typical example is the graphene-like materials that possess the honeycomb structure. And many experimental and theoretical achievements have been made to realize and understand its related novel phenomena \cite{Haldane2008, Hsu2016, Neto2009, Wallace1947}. Due to the underlying geometric symmetries, the honeycomb-structured materials guarantee the existence of Dirac cones in the dispersion surfaces, where two adjacent bands touch each other conically. Such degeneracy carries topology indices, which lead to many topologically protected behaviors \cite{LeeThorp2018}. Recently, Wu and Hu proposed a dielectric material with a generalized honeycomb structure \cite{Wu2015}. This periodic structure has six inclusions inside a hexagonal cell such that the crystal is $\pi/3$-rotation invariant. Based on this structure, high-order Chern topological states are realized \cite{Yves2017}. 
	
	From the mathematical aspects, much progress has been made on honeycomb-structured materials. The pioneering work by Fefferman and Weinstein provided the first rigorous justifications for the existence of Dirac points on the dispersion surfaces of the two-dimensional Schr\"odinger operator with honeycomb lattice potentials \cite{Fefferman2012}. Following their framework,  LeeThorp et al. extended the existence of Dirac points to general elliptic operators with certain honeycomb symmetries \cite{LeeThorp2018}. Besides, other issues are investigated based on honeycomb structures, such as the existence of topologically protected edge states \cite{Fefferman_2016, Fefferman2020}, conical diffractions of wave-packet dynamics associated with Dirac points \cite{Fefferman2013}, and so on. Different types of conical degenerate spectrum points, such as Weyl points in three-dimensional systems, have also been demonstrated with similar approaches \cite{Guo2022}. In discrete setups, tight-binding models are usually used. Many people have extensively investigated them with honeycomb lattices and the connections with their continuous versions. See \cite{Ablowitz2009, Fefferman_2017} to name a few. 
	
	Recently, Ammari and his collaborators successfully applied the integral equation approach to investigate the Helmholtz systems with piece-wise constant material weights. Motivated by the Minnaert resonances for acoustic waves \cite{Ammari_2018, Ammari_2017}, they apply the Gohberg-Sigal theory and layer potential theory \cite{Ammari2018, Ammari2009} to rigorously prove the presence of Dirac points in subwavelength frequencies for honeycomb-structured materials. They paved the way to understanding dielectric or acoustic materials in subwavelength regimes. Many other exciting results are also proved on the edge states\cite{Ammari2020b}, periodically driven systems\cite{Ammari_2022}, and so on.
	
	In this paper, we shall investigate the generalized honeycomb-structured materials proposed by Wu and Hu\cite{Wu2015}. To our knowledge, most results are restricted to numerical simulations of the dispersion surfaces. Very few mathematical analyses have been done. A previous result studied the Schr\"odinger operator equipped with specific potentials possessing symmetries similar to generalized honeycomb structures\cite{Cao2022}. They rigorously proved the existence of double Dirac cones using Fefferman and Weinstein's framework\cite{Fefferman2012}. Dielectric or acoustic materials with generalized honeycomb-structured inclusions have not been mathematically investigated. This paper aims to give a mathematical theory for the dispersion surfaces and associated eigenfunctions of generalized honeycomb-structured materials in the subwavelength regimes. 
	\subsection{Main Results and Contributions}
	In this article, we first characterize the symmetries of generalized honeycomb-structured materials. One certain configuration with additional symmetry will be called super honeycomb-structured material. Two types of deformations that break the additional symmetry are defined consequently. Then we investigate following mathematical problems.
	\begin{itemize}
		\item \emph{Local band gap and eigenfunctions}: \Cref{thm:LocalBandGap} induces the opening of band gap. We first recast the eigenvalue problem equivalently into a characteristic value problem of a boundary integral operator. Motivated by the results in\cite{Ammari2020a}, we then rescale the quasi-periodicity $ \alpha $ to obtain a uniform asymptotic approximation to the single layer potentials in the vicinity of $ \alpha = 0 $. Finally, we relate the periodic capacitance matrix with the original problem asymptotically, as shown in \Cref{thm:EigValAsympt}. By carefully investigating the capacitance matrix, we finally prove the existence of local band gap under both types of deformations. What is more, we numerically verified the absence of the double Dirac cone and the occurrence of a local spectral gap. Corresponding eigenfunctions are obtained to show the topological difference between the two symmetry-breaking modifications to the materials. 
		\item \emph{Existence of double Dirac cones}: \Cref{thm:FourDirac} guarantees the presence of double Dirac point at the center of Brillouin zone. By the additional symmetry of super honeycomb-structured material, a band folding theorem, \Cref{thm:Decomposition}, is proved to decompose layer potentials near $ \alpha=0 $. Then, we can prove the existence of double Dirac cone as a corollary of results in \cite{Ammari2020a}. 
	\end{itemize}
	\begin{remark}
		Topological distinctions which are indicated by \Cref{proposi:EigFuncPer} and its consequences are used to realize topologically protected edge states, which have been experimentally realized \cite{Yang2018}. This will be analyzed in our future work. 
	\end{remark}
	\subsection{Outlines}
	The paper is organized as follows. In \Cref{sec:FormulationAndPreliminary}, we first formulate the problem and review some well-known results on the layer potentials. In \Cref{sec:AsymBandStruc}, the periodic capacitance matrix is defined by asymptotic expansion of Green's function. From this end, we study the asymptotic behaviors of the eigenvalue problem. In \Cref{sec:structure}, we will study the structure of the periodic capacitance matrix. Given the additional symmetries, the existence of a double Dirac point at $ \Gamma $ is rigorously proven in \Cref{sec:ConeStructure}. A direct numerical simulation based on a finite element method is presented in \Cref{sec:NumericalTest}. 
	
	\section{Problem Formulation and Preliminaries} \label{sec:FormulationAndPreliminary}
	In this section, the eigenvalue problem of generalized honeycomb-structured materials is formulated. Following the notations in \cite{Ammari2018}, we briefly review the layer potential theory for periodic inclusions.
	\subsection{Generalized Honeycomb-structured Inclusions}\label{subsec:SuperHoneySymmetry}
	Consider the set of lattice points $ \Lambda \triangleq \{ n_1l_1 + n_2l_2: n_1, n_2\in \mathbb{Z} \} $, where $ l_1, l_2 $ are vectors in $ \mathbb{R}^2 $ given by
	\begin{equation}\label{eqn:BasisVector}
		l_1 = (1/2,-\sqrt{3}/2)^T,\quad l_2 = (1/2,\sqrt{3}/2)^T.
	\end{equation}
	The parallelogram region $\{ sl_1+tl_2:s,t\in [0,1] \}$ can be reformed to the the hexagonal unit cell enclosed by $ A,B,C,D,E,F $ by periodicity, where $ A= (-l_1+l_2)/3, B = (l_1+2l_2)/3,C = (2l_1+l_2)/3, D = -A, E = -B, F = -C $. We denote the hexagonal region by $Y$. See \Cref{fig:Lattice}. \par 
	We introduce the following three matrices $ R $, $ T_x $ and $ T_y $,
	\begin{equation}
		R\triangleq \begin{pmatrix}
			\frac{1}{2} & \frac{\sqrt{3}}{2} \\
			-\frac{\sqrt{3}}{2} & \frac{1}{2}
		\end{pmatrix},\quad 
		T_x \triangleq \begin{pmatrix}
			1 & 0 \\
			0 & -1
		\end{pmatrix},\quad 
		T_y \triangleq \begin{pmatrix}
			-1 & 0 \\
			0  & 1
		\end{pmatrix}.
	\end{equation}
	They correspond to the $ -\pi/3 $-rotation matrix, the reflection with respect to x-axis and y-axis. And we assume that $ D_0 $ is a connected open domain containing the origin $ O $ with piece-wise smooth boundary such that:
	\begin{itemize}
		\item[(i)] $ R^2D_0 = D_0,\quad T_x D_0 = D_0; $
		\item[(ii)] $ \operatorname{diam}(D_0) \triangleq \sup_{x,y\in D_0}\Vert x-y \Vert_{\mathbb{R}^2} < 1/4.$ Here $ \Vert \cdot \Vert_{\mathbb{R}^2} $ denotes the usual Euclidean norm.
	\end{itemize}
	By translations and rotations we define $ \{D_{j}\}_{j=1}^6 $ by 
	\[ D_{1} \triangleq D_0 -\frac{1}{3}(1+\sigma)l_1,
	\quad D_{j} \triangleq R^{j-1}D_1,\;j=2,3,\ldots,6. \] 
	Here and thereafter the parameter $ \sigma $ is chosen such that $ \cup_{j=1}^6 D_j\Subset Y $ and has six connect components. It is easy to see that the point $ P_i $ is contained in $ D_{i} $, where $ P_1 = -(1+\sigma)l_1/3 $, $ P_{j} = R^{j-1}P_{1} $ for $ j=2,3,\ldots 6 $. We let $ D \triangleq \cup_{j=1}^{6} D_{j} $. And its periodic tiling $ D_{per} $ can be defined by 
	\begin{equation}
	    D_{per} = \{ x+n_1l_1 + n_2l_2: x\in D, n_1,n_2\in \mathbb{Z} \}.
	\end{equation}\par 
	In this article, the periodic inclusions $D_{per}$ will be called \emph{generalized honeycomb-structured inclusions}. Specifically for $ \sigma=0 $, the inclusions $ D_{per} $ will be called the \emph{super honeycomb-structured inclusions}. When $ \sigma < 0(>0)$, we will say that the inclusions are \emph{contracted}(\emph{dilated}). An example of $D_{per} $ when $\sigma=0$ is illustrated in \Cref{fig:Lattice}. While in \Cref{fig:LatticePert} we illustrate the inclusions when they are either contracted or dilated. The corresponding super honeycomb-structured inclusions in unit cell are plotted in dash lines. 
	\begin{figure}[htbp] 
		\centering
		\vspace{-0.3cm}
		\includegraphics[width=0.4\textwidth]{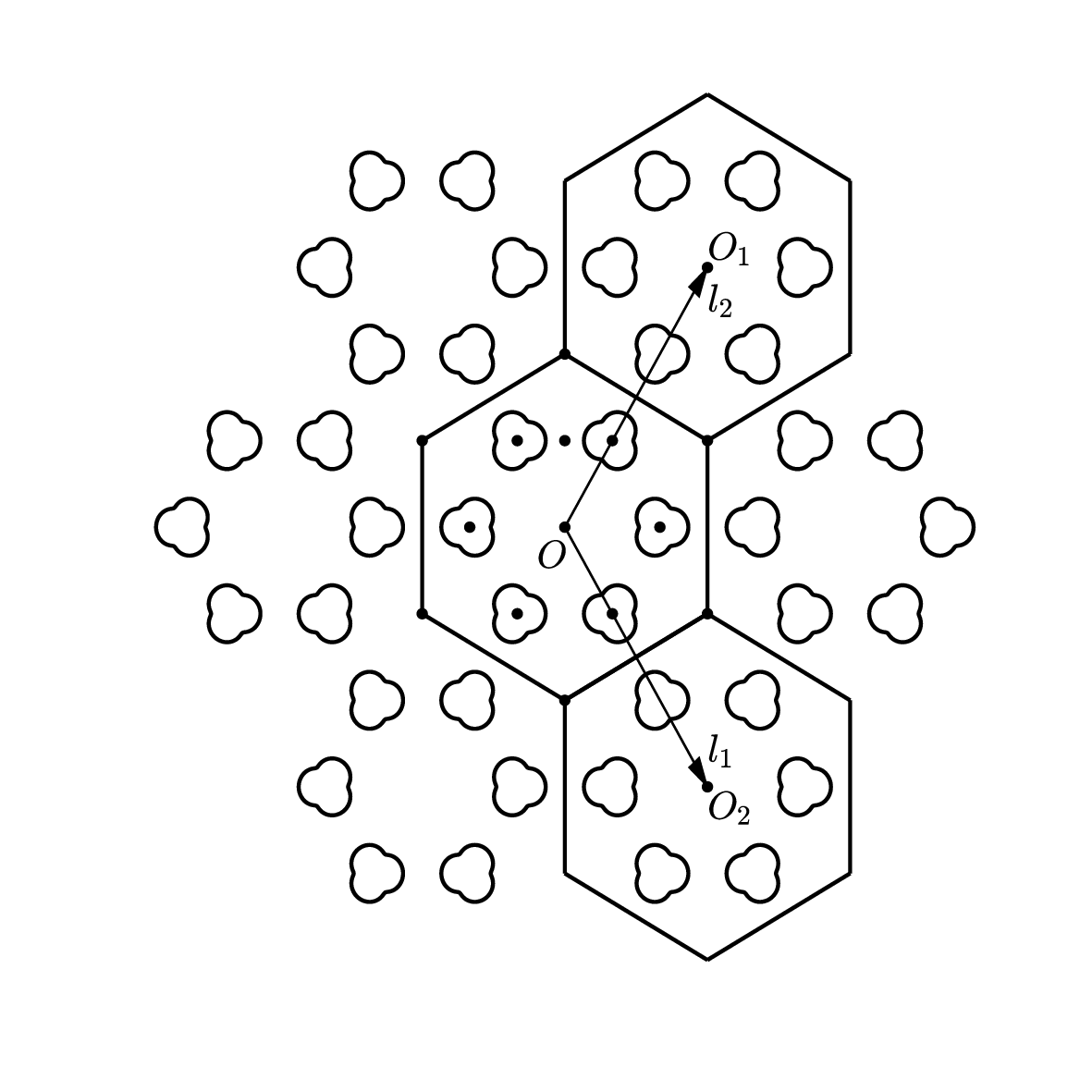}
		\hspace{-0.3cm}
		\includegraphics[width=0.45\textwidth]{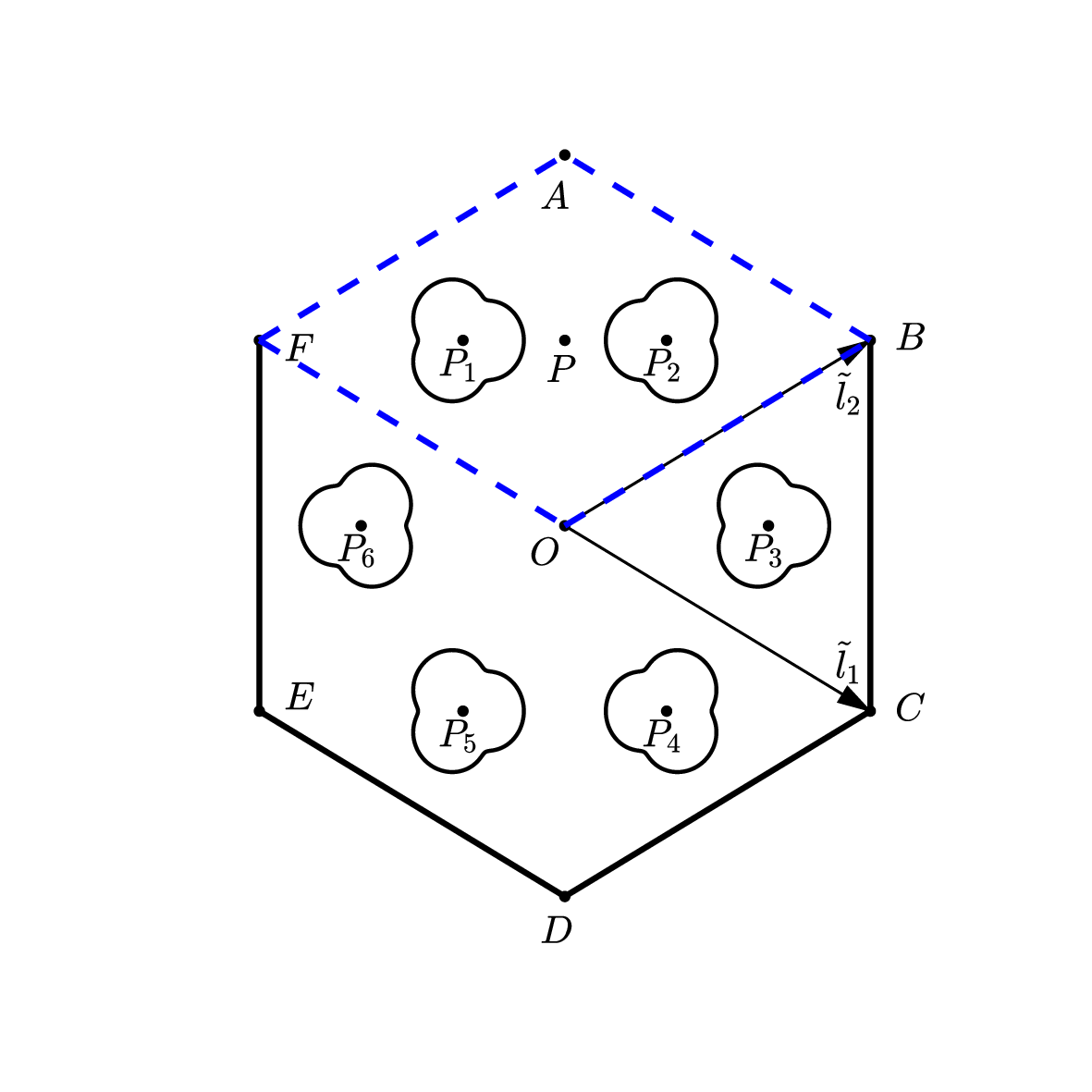}
		\vspace{-0.4cm}
		\caption{Left: honeycomb-structured inclusions. The unit cell $ Y $ is enclosed in black bold line. Right: illustration of the unit cell $Y$. The region $\tilde{Y}$ is enclosed in blue dash lines. }
		\label{fig:Lattice}
	\end{figure} 
	\begin{figure}[htbp] 
		\centering
		\vspace{-0.7cm}
		\includegraphics[width=0.26\textwidth]{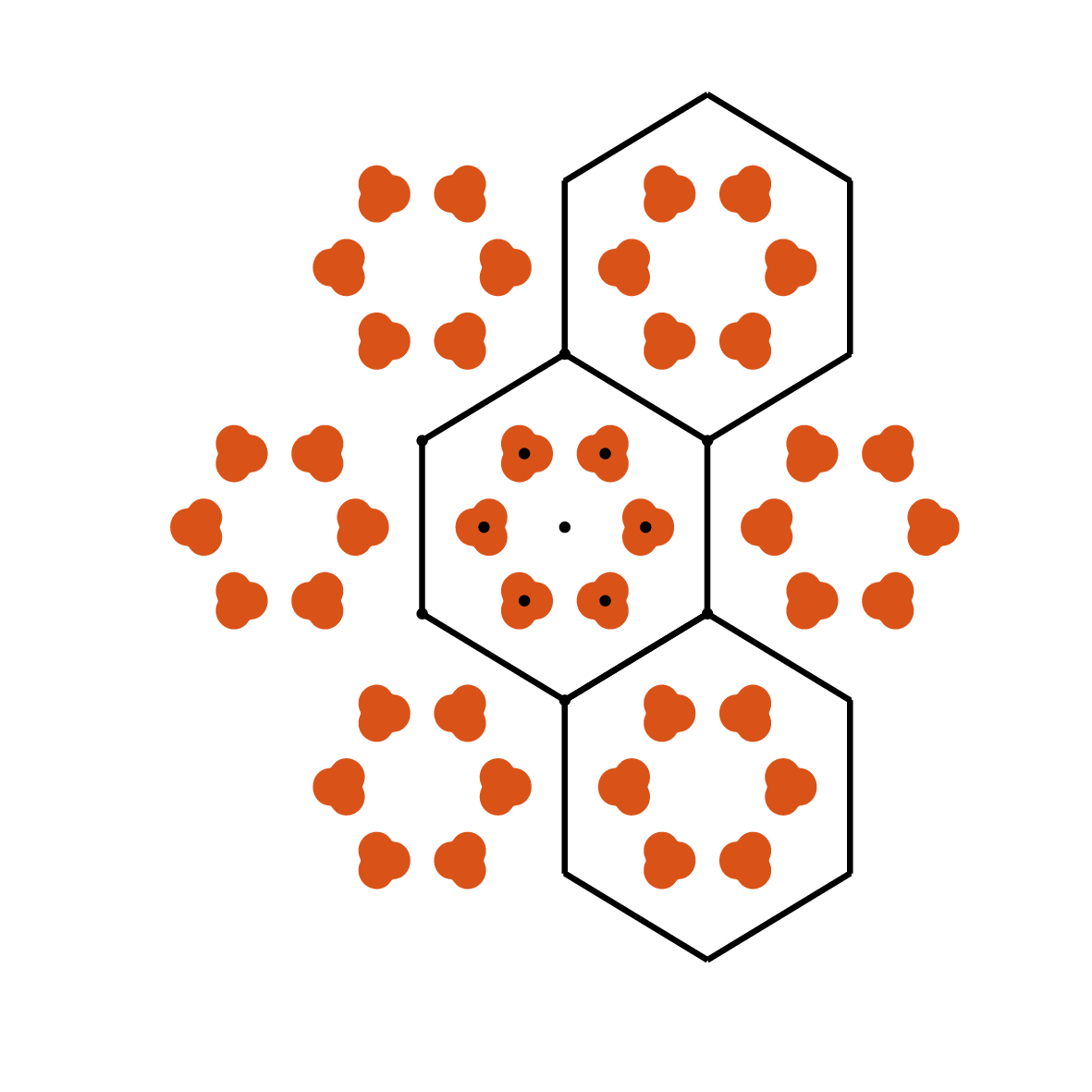}
		\hspace{-0.5cm}
		\includegraphics[width=0.26\textwidth]{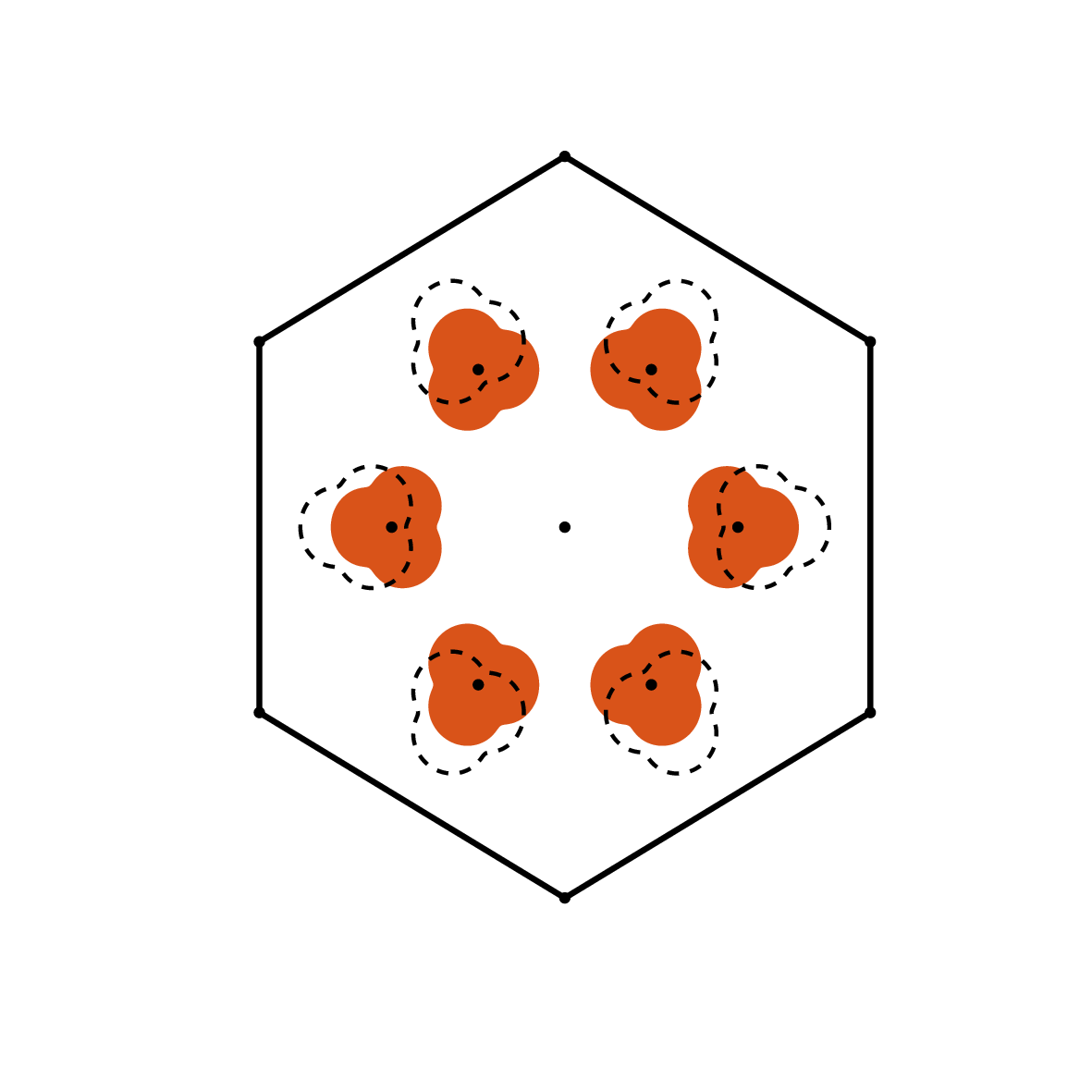}
		\hspace{-0.5cm}
		\includegraphics[width=0.26\textwidth]{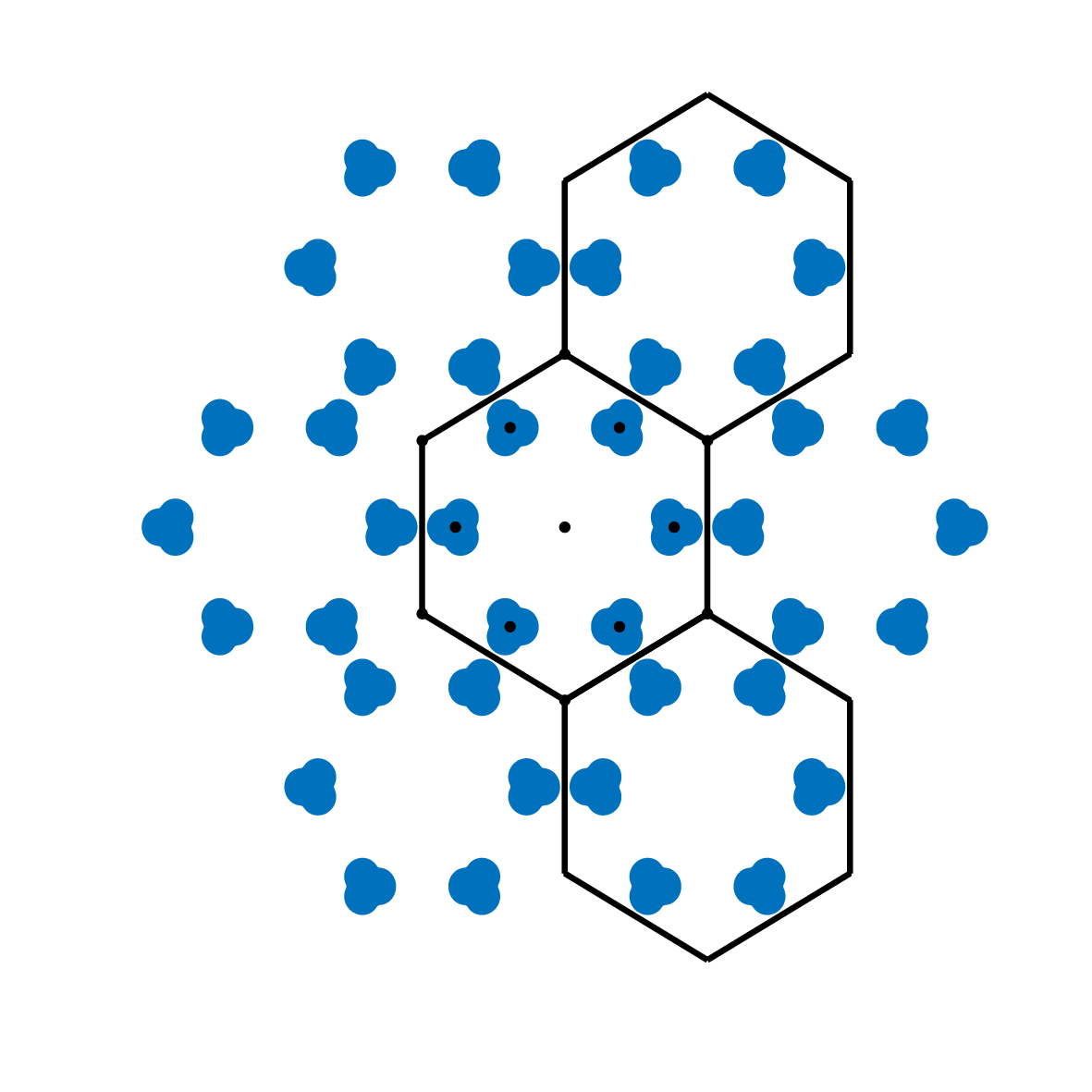}
		\hspace{-0.5cm}
		\includegraphics[width=0.26\textwidth]{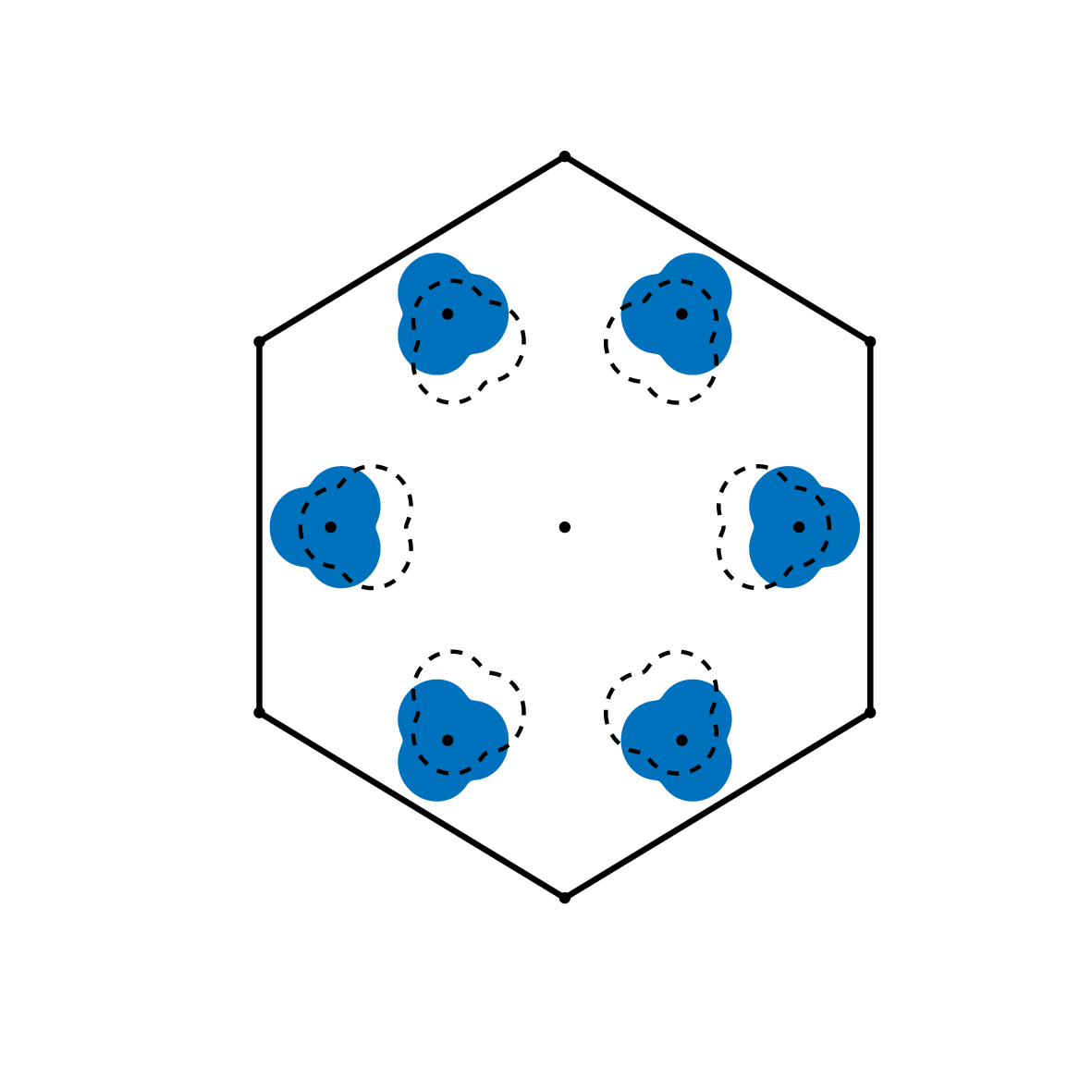}
		\vspace{-0.4cm}
		\caption{Left: illustrations of the generalized honeycomb-structured inclusions and the unit cell when the inclusions are contracted. Here $ \sigma  $ is negative. Right: illustration of the generalized honeycomb-structured inclusions and the unit cell when the inclusions are dilated. Here $ \sigma  $ is positive. The corresponding super honeycomb-structured inclusions are plotted in dash lines.}
		\label{fig:LatticePert}
	\end{figure}\par 
	Through simple calculations, one can verify that the generalized honeycomb-structured inclusions is invariant under the $ \pi/3 $-rotation at $ O $ and the reflection with respect to y-axis, which is equivalent to $ RD_{per} = D_{per}, T_yD_{per} = D_{per} $. \par 
	The dual lattice of $ \Lambda $ is defined as $ \Lambda^{\prime} \triangleq \{ n_1k_1+n_2k_2:n_1,n_2\in \mathbb{Z} \}, $ with vectors $ k_1 $ and $ k_2 $ satisfying
	\begin{equation}\label{eqn:DualRelation}
		k_i\cdot l_{j} = 2\pi \delta_{ij},\quad i,j = 1,2.
	\end{equation}
	The Brillouin zone is defined by $ Y^{\prime}\triangleq \mathbb{R}^2/\Lambda^{\prime} $. For convenience we choose the representation $ Y^{\prime} \simeq \{ sk_1+tk_2:s,t\in [0,1) \}. $
	\subsection{Problem Formulation}
	Given the geometry properties, we now define the eigenvalue problem in the generalized honeycomb-structured crystal, 
	\[ \nabla \cdot ( \frac{1}{\rho(x)} \nabla u(x) ) +\frac{\omega^2}{\kappa(x)} u(x)=0,\quad  u(x)\in \mathbb{L}^2(\mathbb{R}^2). \]
	Here the material weights $ \rho(x),\kappa(x) $ are given by 
	\[ \rho(x) \triangleq \rho_{0}\chi_{_{\mathbb{R}^2\backslash D_{per}}}(x) + \rho_{1}\chi_{_{D_{per}}}(x),\quad \kappa(x)\triangleq \kappa_{0}\chi_{_{\mathbb{R}^2\backslash D_{per}}}(x) + \kappa_{1}\chi_{_{D_{per}}}(x), \]
	where $ \chi_{_{D_{per}}} $ and $ \chi_{_{\mathbb{R}^2\backslash D_{per}}} $ are the characteristic functions of $ D_{per} $ and $ \mathbb{R}^2\backslash D_{per} $. And $ \rho_{0},\kappa_{0} $ and $ \rho_{1},\kappa_{1} $ denote the densities and bulk moduli outside and inside the inclusions, respectively. We introduce
	\[
	v_0\triangleq  \sqrt{\frac{\kappa_{0}}{\rho_{0}}},\;v_1\triangleq\sqrt{\frac{\kappa_{1}}{\rho_{1}}},\;k_0 \triangleq \frac{\omega}{v_0},\;k_1 \triangleq \frac{\omega}{v_1},\;\delta \triangleq \frac{\rho_{1}}{\rho_{0}}.
	\]
	By Floquet-Bloch decomposition, the eigenvalue problem in $\mathbb{L}^2(\mathbb{R}^2)$ is reduced to the quasi-periodic eigenvalue problem for all $\alpha\in Y'$. Combining the jump relations, we can formulate the $\alpha$-quasi-periodic problem for $ \alpha\in Y' $:
	\begin{gather}\label{eqn:ProblemFormulation}
		\left\{\begin{aligned}
			&\Delta u + k_0^2 u = 0,\quad x\in \mathbb{R}^2\backslash \overline{D_{per}},\\
			&\Delta u + k_1^2 u = 0,\quad x\in D_{per}, \\
			&u|_{+}-u|_{-}=0,\quad y\in \partial D_{per},\\
			&\delta\left.n_y\cdot \nabla u \right|_{+} - \left.n_y\cdot \nabla u\right|_{-}=0,\quad y\in \partial D_{per},\\
			&u(x+l) = \eu^{\iu\alpha\cdot l}u(x),\quad \forall x\in \mathbb{R}^2,l\in \Lambda.
		\end{aligned}\right.
	\end{gather}   
	Here, $ n_y\cdot \nabla u $ denotes the normal derivative of $ u $ at $ y\in \partial D_{per} $, and the subscripts $ + $ and $ - $ denotes the limit from outside and inside of $ D_{per} $. For any $ \alpha\in Y^{\prime} $, the eigenvalue problem \eqref{eqn:ProblemFormulation} has discrete eigenvalues $ \{\omega^2_n(\alpha,\sigma,\delta)\} $ which are non-decreasing and tend to infinity as $ n\to +\infty $.  \par 
	The parameter $ \delta $ denotes the contrast in the density and we assume that $ \delta $ is small, while the wave speed are comparable and of $ \mathcal{O}(1) $, i.e.,
	\[ 0<\delta \ll 1\quad \text{and}\quad v_0,v_1 = \mathcal{O}(1). \]
	In this article, we say that a frequency $ \omega $ is a \emph{subwavelength frequency} if $ \omega^2= \mathcal{O}(\delta) $. And the problem of interest is the subwavelength frequency when the quasi-periodicity vector $ \alpha $ is close to the $ \Gamma $ point, namely the point $ \alpha=0 $. We are also interested in the behavior of band structures when the inclusions are contracted or dilated. 
	\begin{remark}
	Notice that the subwavelength regime is different from the tight binding regime\cite{Ablowitz_2012a,Fefferman_2017} or when the density $ \rho(x) $ is very small outside the inclusion\cite{Cassier2021}.
	\end{remark}
	\subsection{Green's Functions and Layer Potentials}
	Consider the quasi-periodic Green's function $G^{\alpha,\omega}$ defined by the spectral representation
	\begin{equation}\label{eqn:QuasiGreensFunc}
		G^{\alpha,\omega}(x) = \frac{1}{|Y|}\sum_{q\in \Lambda^{\prime}} \frac{\eu^{\iu(\alpha + q)\cdot x}}{\omega^2 - |\alpha+q|^2},
	\end{equation} 
	where $ \omega \neq |\alpha + q| $ for all $ q\in \Lambda' $. If both $ \omega $ and $ \alpha $ are zero, we define the periodic Green's function $G^{0,0}$ by 
	\begin{equation}\label{eqn:PeriodicGreenFunc}
		G^{0,0}(x) = \frac{1}{|Y|}\sum_{q\in \Lambda^{\prime}\backslash\{0\}} -\frac{\eu^{\iu q\cdot x}}{ |q|^2}.
	\end{equation}
	Now we define the quasi-periodic single layer potential of the density function $\phi\in \mathbb{L}^2(\partial D) $ as 
	\begin{equation}
		\mathcal{S}_{D}^{\alpha,\omega}[\phi] \triangleq \int_{\partial D} G^{\alpha,\omega}(x-y)\phi(y)d\sigma(y).
	\end{equation}
	Then the following well-known jump relations are implied directly \cite{Ammari2018}
	\begin{equation}\label{eqn:JumpSingle}
		\left.n_y\cdot \nabla\mathcal{S}_{D}^{\alpha,\omega}[\phi]\right|_{\pm} = \pm\frac{\phi(y)}{2}+(\mathcal{K}^{-\alpha,\omega}_{D})^{\ast}[\phi](y), 
	\end{equation}
	where the Neumann-Poincar\'e operator $ (\mathcal{K}_{D}^{-\alpha,\omega})^{\ast} $ is defined by 
	\begin{equation}
		(\mathcal{K}_{D}^{-\alpha,\omega})^{\ast}[\varphi](x) \triangleq \operatorname{p.v.}\int_{\partial D} n_{x} \cdot \nabla G^{\alpha,\omega}(x-y)\varphi(y)d\sigma(y),\quad x\in \partial D.
	\end{equation}
	Its $ \mathbb{L}^2(\partial D) $ adjoint operator $ \mathcal{K}_{D}^{\alpha,\omega} $ is given by
	\begin{equation}
		\mathcal{K}_{D}^{\alpha,\omega}[\varphi](x) \triangleq \operatorname{p.v.}\int_{\partial D} n_{y} \cdot \nabla G^{\alpha,\omega}(x-y)\varphi(y)d\sigma(y),\quad x\in \partial D.
	\end{equation}
	One can prove that the single layer potential is invertible when it is regarded as an operator $ \mathcal{S}_{D}^{\alpha,\omega}:\mathbb{L}^2(\partial D)\to \mathbb{H}^{1}(\partial D) $ when $ \omega $ is small enough and $ \omega\neq|\alpha+q| $ for all $ q\in \Lambda^\prime $, see Lemma 3.1 in \cite{Ammari2020a}. 
	\subsection{Integral Equation Formulation}
	By layer potential theory, the original problem \eqref{eqn:ProblemFormulation} can be recast into an integral equation problem. We represent the solution by single layer potentials if $ (\omega,\alpha) $ is nonzero:
	\begin{gather}\label{eqn:RepresentationLayer}
		u(x) = \left\{\begin{aligned}
			&\mathcal{S}^{\alpha,k_1}_D[\phi](x),\quad x\in D,\\
			&\mathcal{S}^{\alpha,k_0}_D[\psi](x),\quad x\in Y\backslash \overline{D},
		\end{aligned}\right.
	\end{gather}
	Here $ \phi,\psi\in \mathbb{L}^{2}(\partial D) $ and are determined by following equations:
	\begin{gather}\label{eqn:BoundaryMatch}
		\left\{ \begin{aligned}
			&\mathcal{S}^{\alpha,k_1}_D[\phi](y) = \mathcal{S}_D^{\alpha,k_0}[\psi](y)\\
			&-\frac{\phi(y)}{2} + (\mathcal{K}^{-\alpha,k_1}_{D})^{\ast}[\phi](y) = \delta\left( \frac{\psi(y)}{2} + (\mathcal{K}^{-\alpha,k_0}_{D})^{\ast}[\psi](y) \right)
		\end{aligned}\right.,\quad y\in \partial D. 
	\end{gather}\par 
	If $ (\omega,\alpha) $ is zero, we will prove in the appendix that the harmonic periodic solutions can be represented by single layer potential up to a constant, see the discussions in \Cref{apsec:RepresentSol}. So the usual procedures (see \cite{Ammari2020}) to investigate the band structure near $\alpha = 0$ fail due to the lack of invertibility of $\mathcal{S}^{0,0}_{D}$. The singularity will also arise when we consider the band structures in subwavelength regime. More importantly, we have to first justify the representation \eqref{eqn:RepresentationLayer} of the eigenfunctions corresponding to subwavelength eigenvalues.  \par 
	Fortunately, again by the discussions in \Cref{apsec:RepresentSol}, one only has to take care of constant functions. And the characteristic value problem of the operator $ \mathcal{A}^{\alpha ,\omega}_{\delta}: (\mathbb{L}^2(\partial D))^2\to (\mathbb{L}^2(\partial D))^2$ given by 
	\begin{equation}
		\mathcal{A}^{\alpha ,\omega}_{\delta}\begin{pmatrix}
			\phi\\ \psi  
		\end{pmatrix}\triangleq \begin{Bmatrix}
			\mathcal{S}^{\alpha,k_1  }_{D}[\phi]    -\mathcal{S}^{\alpha,k_0}_{D}[\psi]\\
			-\frac{1}{2}\phi + \left(\mathcal{K}^{-\alpha,k_1}_{D}\right)^{\ast}[\phi]  -\delta\left[\frac{1}{2}\psi + \left(\mathcal{K}^{-\alpha,k_0}_{D}\right)^{\ast}[\psi]\right] 
		\end{Bmatrix},
	\end{equation}
	is still able to represent the corresponding eigenvalues of \eqref{eqn:ProblemFormulation}. Here notice that $ \mathcal{A}^{\alpha,\omega}_{\delta} $ is a function of $ \omega $, since $ k_0 $ and $ k_1 $ are functions of $ \omega $. With this definition, \eqref{eqn:BoundaryMatch} is equivalent to 
	\begin{equation}
		\mathcal{A}^{\alpha ,\omega}_{\delta}\begin{pmatrix}
			\phi\\ \psi  
		\end{pmatrix} = \begin{pmatrix}
			0\\ 0 
		\end{pmatrix}.
	\end{equation}
	By the properties of elliptic eigenvalue problem, the above integral equations has non-trivial solutions for some discrete frequencies $ \omega $. And they can be viewed as the characteristic values of the operator-valued analytic function $ \mathcal{A}^{\alpha ,\omega}_{\delta}:\mathbb{C}\to [(\mathbb{L}^2(\partial D))^2\to (\mathbb{L}^2(\partial D))^2]$. 
	
	\section{Asymptotic Behavior of Band Structure}\label{sec:AsymBandStruc}
	This section is devoted to characterizing the band structure in subwavelength regime of generalized honeycomb-structured inclusions $ D $ for arbitrary $ \sigma $. 
	We will asymptotically determine the subwavelength frequencies by the eigenvalues of capacitance matrix $\mathbf{C}^{0}$, which will be defined later. We will prove the following result:
	\begin{theorem}\label{thm:EigValAsympt}
	 	The value $ \{\omega_{n}^2(\alpha,\sigma,\delta)\}_{n=1}^{6} $ are approximated by   
	 	\[
	 	\omega_{n}^2(\alpha,\sigma,\delta)= \frac{6\delta \lambda_{n}v_1^2}{|D|} + \mathcal{O}(\delta^2),
	 	\]
	 	uniformly for $ |\alpha|< \min(k_1^2,k_0^2) $. Here $ |D| $ is the area of the inclusions in the unit cell, and $ \lambda_{n} $ are the eigenvalues of the periodic capacitance matrix $ \mathbf{C}^{0} $, which are arranged in ascending order.
	\end{theorem}
	\subsection{Asymptotic Expansion of Green's Function}
	In order to study the asymptotic behavior near $ \alpha=0 $, we set $\alpha = \omega^2\alpha_0$ for some $ \alpha_0 \in Y^{\prime} $ independent of $ \omega $. When $\omega\to 0$ and $ \alpha\neq 0 $, we have, from \eqref{eqn:QuasiGreensFunc},
	\begin{equation}\label{eqn:AsymExpZero}
		G^{\omega^2\alpha_0,\omega}(x) = \frac{1}{|Y|}\Big(\frac{1}{\omega^2} + \iu \alpha_0 \cdot x + |\alpha_0|^2\Big) +  G^{0,0}(x) +\omega^2G^{\alpha_0}_2 + o(\omega^2).
	\end{equation}
	Here $ G^{\alpha_0}_2 $ is defined as the second order term in the above expansion. From this we obtain an $ \mathcal{O}(\omega^2) $ approximation of $ \hat{S}^{\omega^2\alpha_0,\omega}_D $: 
	\begin{gather} \label{eqn:SingleLayerHatOperator}
		\begin{aligned}
			\hat{S}^{\omega^2\alpha_0,\omega}_D[\varphi](x) &\triangleq \mathcal{S}_D^{0,0}[\varphi](x) + \frac{1}{|Y|}\left(\frac{1}{\omega^2}+|\alpha_0|^2\right)\int_{\partial D}\varphi(y) d\sigma(y) \\&\quad+ \frac{\iu}{|Y|}\int_{\partial D}\alpha_0\cdot(x-y)\varphi(y)d\sigma(y).
		\end{aligned}
	\end{gather}
	Here $ \mathcal{S}^{0,0}_{D} $ denotes the periodic single layer potential associated to $ G^{0,0} $. Although the operator $ \mathcal{S}^{0,0}_{D} $ is not invertible in general, it can be proved that for small but nonzero $\omega$, the operator $ \hat{S}^{\omega^2\alpha_0,\omega}_D $ is invertible \cite{Ammari2020a}. However, the kernel of periodic single layer potential can be characterized following the same spirit of \cite{Ammari_2021,Ammari2020a}:
	\begin{lemma}\label{lem:uniqueness}
		If $\mathcal{S}^{0,0}_{D}[\varphi] = K \chi_{_{\partial D}}$ for some constant $ K $ and some $ \varphi\in\mathbb{L}^2_0(\partial D)\triangleq \{ \phi:\phi\in \mathbb{L}^2(\partial D),\int_{ \partial D}\phi d\sigma = 0\} $, then $ \varphi \equiv0 $.
	\end{lemma}
	\begin{proof}
		By the uniqueness result of the homogeneous Dirichlet problem (see, for example \cite{Riva2021} or \Cref{thm:Uniqueness}), the function $\mathcal{S}^{0,0}_{D}[\varphi] $ is a constant. Then by \eqref{eqn:JumpSingle} we obtain 
		\[ \varphi(y) = n_y\cdot \mathcal{S}^{0,0}_{D}[\varphi]|_{+} - n_y\cdot \mathcal{S}^{0,0}_{D}[\varphi] |_{-}=0.\]
	\end{proof}
	In the asymptotic expansion formula \eqref{eqn:AsymExpZero}, it is clear that the operator $ \hat{S}_D^{\omega^2\alpha_0,\omega} $ has a $ \mathcal{O}(\omega^{-2}) $ singularity near zero. However, the inverse operator $ \left(\hat{S}^{\omega^2\alpha_0,\omega}_{D}\right)^{-1} $ behaves well near the zero point.
	\begin{lemma}\label{lem:holoOp}
		Given $ \alpha_0\in Y^{\prime} $ such that $ |\alpha_0|<1 $, the operator-valued function $ (\hat{S}^{\omega^2\alpha_0,\omega}_{D})^{-1} $ is holomorphic with respect to $ \omega $ in a neighborhood of $ \omega=0 $. 
	\end{lemma}
	\begin{proof}
		From the definition \eqref{eqn:SingleLayerHatOperator}, we know that $ \hat{S}^{\omega^2\alpha_0,\omega}_{D} $ is a meromorphic \newline operator-valued function of $ \omega $ with a pole at $\omega=0$. 
		Now assuming $ \left(\hat{S}^{\omega^2\alpha_0,\omega}_{D}\right)^{-1} $ is singular as $ \omega \to 0 $, there exist some $\phi_{\omega}$ depending on $\omega$ such that $ ||  \phi_{\omega}||_{\mathbb{L}^2(\partial D)} =1$ while $ || \hat{S}^{\omega^2\alpha_0,\omega}_{D}[\phi_{\omega}] ||_{\mathbb{H}^1(\partial D)} =\mathcal{O}(\omega)$. We suppose the $\mathcal{O}(1)$ part of $ \phi_{\omega} $ is $ \phi_0 $, which is independent on $\omega$. Then from \eqref{eqn:SingleLayerHatOperator}, we have uniformly in $ \omega $, 
		\[
		\hat{S}^{\omega^2\alpha_0,\omega}_D[\phi_{\omega}](x) = \frac{1}{\omega^2|Y|}\int_{\partial D}\phi_{\omega}d\sigma(y)+\mathcal{O}(1).
		\] 
		When $ |\omega|\ll 1 $, we obtain 
		\[\int_{\partial D} \phi_0\,d\sigma(y) = 0. \]
		Thus we have, $ \hat{S}^{\omega^2\alpha_0,\omega}_D[\phi_0] =K\chi_{_{\partial D}}$ for some constant $ K $. It follows from \Cref{lem:uniqueness} that $\phi_0=0$, which contradicts the fact that $ ||\phi_{\omega}||_{\mathbb{L}^2(\partial D)} = \mathcal{O}(1) $.
	\end{proof}
	Now for fixed $\alpha_0\in Y^{\prime}$, $ \mathbb{H}^1(\partial D)\ni f_R \sim \mathcal{O}(1) $ and consider the function
	\begin{equation}\label{eqn:defini_capacitance}
		\hat{\phi}_{j}^{\omega^2\alpha_0,\omega} = \left(\hat{S}^{\omega^2\alpha_0,\omega}_{D}\right)^{-1}[\chi_{_{\partial D_j}}+\omega^2f_R]. 
	\end{equation}
	From \Cref{lem:holoOp}, we have the following expansion 
	\begin{equation}\label{eqn:expansion}
		\textstyle\hat{\phi}_{j}^{\omega^2\alpha_0,\omega} = \phi^{\alpha_0}_{j,0} +\omega\hat{\phi}^{\alpha_0}_{j,1}+\omega^2\hat{\phi}^{\alpha_0}_{j,R},
	\end{equation}
	as $\omega\to 0$ for some $ \phi^{\alpha_0}_{j,0},\hat{\phi}^{\alpha_0}_{j,1}, \hat{\phi}_{j,R}^{\alpha_0}\in \mathbb{L}^2(\partial D)  $. Here $ \hat{\phi}_{j,R}^{\alpha_0}\sim \mathcal{O}(1) $ and may depend on $ \omega $. We first state and prove the main properties of $ \phi_{j,0}^{\alpha_0},\hat{\phi}_{j,1}^{\alpha_0},\hat{\phi}_{j,R}^{\alpha_0} $ determined by \eqref{eqn:defini_capacitance} and \eqref{eqn:expansion}. 
	\begin{proposi}\label{proposi:asymptBehav}
		For every fixed $\alpha_0\in Y^{\prime}$ and $ f_r\in \mathbb{H}^1(\partial D) $, we have the following properties:
		\begin{itemize}
			\item[(i)] The zeroth order term $\phi^{\alpha_0}_{j,0}\in \mathbb{L}^2_{0}(\partial D)$ is independent of $\alpha_0$ and $ f_r $ and satisfies 
			\[ \mathcal{S}^{0,0}_{D}[\phi_{j,0}^{\alpha_0}] = \chi_{_{\partial D_j}}-\frac{1}{6}\chi_{_{\partial D}}; \]
			\item[(ii)] The first order term $\hat{\phi}^{\alpha}_{j,1}$ vanishes, which is equivalent to $\hat{\phi}^{\alpha}_{j,1} \equiv 0$;
			\item[(iii)] The following identity holds:
			\begin{equation}\label{eqn:second_order_term}
				\int_{\partial D}\hat{\phi}^{\alpha_0}_{j,R}d\sigma(y) - \iu\int_{\partial D}(\alpha_0\cdot y)\phi_{j,0}^{\alpha_0}\,d\sigma(y) = \frac{|Y|}{6}.
			\end{equation}
		\end{itemize}
	\end{proposi}
	\begin{proof}
		Since $\hat{S}^{\omega^2\alpha_0,\omega}_{D}[\hat{\phi}_{j}^{\omega^2\alpha_0,\omega}]$ is bounded as $\omega\to 0$, the singular part of $ \hat{S}^{\omega^2\alpha_0,\omega}_{D} $ should vanish on $\phi_{j,0}^{\alpha_0}$ and $\hat{\phi}^{\alpha_0}_{j,1}$ for any $ \alpha_0\in Y^{\prime} $. Therefore, we conclude that \newline $ \phi^{\alpha_0}_{j,0},\hat{\phi}^{\alpha_0}_{j,1}\in \mathbb{L}^2_0(\partial D) $.
		From the definition \eqref{eqn:defini_capacitance}, we have $ \hat{S}^{\omega^2\alpha_0,\omega}_{D}[\hat{\phi}_j^{\omega^2\alpha_0,\omega}] = \chi_{_{\partial D_j}} +\omega^2f_R$. Taking the $\mathcal{O}(1)$ terms gives us the following identity:
		\begin{equation}
			\chi_{_{\partial D_j}} = \mathcal{S}^{0,0}_{D}[\phi^{\alpha_0}_{j,0}] -\frac{\iu}{|Y|}\int_{\partial D}(\alpha_0\cdot y)\phi^{\alpha_0}_{j,0}d\sigma(y) + \frac{1}{|Y|}\int_{\partial D} \hat{\phi}^{\alpha_0}_{j,R}\,d\sigma(y)\label{eqn:zero_order_expansion}.
		\end{equation}
		And thus $\mathcal{S}^{0,0}_{D}[\phi^{\alpha_0}_{j,0}] = \chi_{_{\partial D_j}}+\tilde{K}_j\chi_{_{\partial D}}$ for some constant $\tilde{K}_j$. By uniqueness result \Cref{thm:Uniqueness}, we conclude that $ \mathcal{S}^{0,0}_D[\phi_{j,0}^{\alpha_0}] $ equals to constant in each inclusion $ D_{k} $. As a result, we have from \eqref{eqn:JumpSingle}
		\[ -\frac{1}{2}\phi_{j,0}^{\alpha_0} + (\mathcal{K}^{0,0}_{D})^{\ast}[\phi_{j,0}^{\alpha_0}]=\left.n_y\cdot\nabla \mathcal{S}^{0,0}_D[\phi_{j,0}^{\alpha_0}]\right|_{-}=0. \]
		From \Cref{approposi:KernelKastOp}, we have $ \tilde{K}_j=-1/6 $ for all $\alpha_0\in Y^{\prime} $. So we have proved (i).\par 
		To prove (ii), we first substitute $ \hat{\phi}_{j}^{\omega^2\alpha_0,\omega}$ into \eqref{eqn:defini_capacitance} and take the $ \mathcal{O}(\omega) $ terms. 
		\begin{equation}
			\mathcal{S}^{0,0}_{D}[\hat{\phi}^{\alpha_0}_{j,1}] =  \frac{\iu}{|Y|}\int_{\partial D}(\alpha_0\cdot y)\hat{\phi}^{\alpha_0}_{j,1}d\sigma(y) = K\chi_{_{\partial D}},
		\end{equation} 
		for some constant $ K=K(\omega) $. By \Cref{lem:uniqueness}, the first order term in \eqref{eqn:expansion} equals to zero, which finishes the proof of (ii). \par 
		As for (iii), we immediately draw the conclusion from \eqref{eqn:zero_order_expansion} and (i).
	\end{proof}
	Now we can define 
	\begin{equation}\label{eqn:capacitanceMatDef}
		C_{jk}^{0} \triangleq -\int_{\partial D_k}\phi_{j,0}^{\alpha_0} \,d\sigma(y),
	\end{equation}
	And we call the matrix $ \mathbf{C}^{0}=(C^{0}_{jk})_{j,k=1}^6 $ the \emph{periodic capacitance matrix}. From \Cref{proposi:asymptBehav} we know that the periodic capacitance matrix is independent of $ \alpha_0,\omega $ and $ f_r $. We can also conclude that $ \phi_{j,0}^{\alpha} $ and $ \hat{\phi}^{\alpha}_{j,1} $ are independent of $ \omega $. In the rest of the paper, we omit the superscript in $ \phi_{j,0}^{\alpha_0} $ and write $ \phi_{j,0} $.
	\begin{remark}\label{rmk:RepresentationCapacity}
		One can prove by Green's identity that the elements of periodic capacitance matrix can be equivalently defined by 
		\begin{equation}
			C_{jk}^{0} = \int_{Y\backslash D} (\nabla \mathcal{S}^{0,0}[\phi_{j,0}] \cdot \nabla \mathcal{S}^{0,0}[\phi_{k,0}])dx, \quad j,k = 1,2,\ldots 6.
		\end{equation}
	\end{remark}
	
	\subsection{Asymptotic Behavior of Band Structure}
	Given the results before, we will now prove \Cref{thm:EigValAsympt} concerning with the asymptotic behavior of subwavelength frequencies. Before giving the detailed proof, we first give several lemmas.
	\begin{lemma}\label{lem:FOperator}
		The operator $F^{\alpha_0}:\mathbb{L}^{2}(\partial D) \to \mathbb{L}^{2}(\partial D)$ defined by 
		\begin{equation}\label{eqn:FOperatorDef}
			F^{\alpha_0}[\phi](y) \triangleq -\frac{\phi}{2}(y)+(\mathcal{K}^{0,0}_D)^\ast[\phi](y) + \frac{\iu}{|Y|}\alpha_0\cdot n_y\int_{\partial D}\phi(z) d\sigma(z),\quad\forall \alpha_0 \in Y^{\prime},
		\end{equation}
		is a Fredholm operator of index zero, with $\ker(F^{\alpha_0}) = \ker(-\frac{1}{2}\operatorname{Id} + (\mathcal{K}^{0,0}_{D})^{\ast}) $.
	\end{lemma}
	\begin{proof}
		The operator $ F^{\alpha_0} $ is a rank-1 perturbation of the Fredholm operator $ \ker(-\frac{1}{2}\operatorname{Id} + (\mathcal{K}^{0,0}_{D})^{\ast}) $, so it is a Fredholm operator of index zero. \par 
		Noticing a simple fact, 
		\begin{equation}\label{eqn:NormalBdry}
			\int_{ \partial D_k}(\alpha_0\cdot n_y)d\sigma(y) = \int_{D_k}\operatorname{div}(\alpha_0)dx=0,
		\end{equation}
		and integrating both sides of \eqref{eqn:FOperatorDef} on $ \partial D $, we have, for any given $\phi\in \ker(F^{\alpha_0})$, 
		\[
		\int_{\partial D}\left(-\frac{1}{2}\phi+(\mathcal{K}^{0,0}_D[\phi])^\ast\right)d\sigma(y) = 0.
		\]
		By the mapping property given in Proposition 12.15 in \cite{Riva2021}, we have $\phi\in \mathbb{L}^2_0(\partial D)$. Therefore, $\phi$ satisfies 
		\[ -\frac{\phi}{2}+(\mathcal{K}^{0,0}_D)^\ast[\phi]=0, \]
		which is equivalent to say $\phi\in\ker(-\frac{1}{2}\operatorname{Id} + (\mathcal{K}^{0,0}_{D})^{\ast})$. \par 
		And for every $ \phi\in\ker(-\frac{1}{2}\operatorname{Id} + (\mathcal{K}^{0,0}_{D})^{\ast}) $, we have, by Proposition 12.15 in \cite{Riva2021}, $ \phi\in \mathbb{L}^2_0(\partial D) $. Substituting this fact into \eqref{eqn:FOperatorDef}, we know that $ \phi \in \ker(F^{\alpha_0}) $.
	\end{proof}
	\begin{lemma}
		The operator $ \mathcal{K}^{\alpha_0,0}_{D,2} $ defined by 
		\[
		\mathcal{K}^{\alpha_0,0}_{D,2}[\phi] \triangleq \int_{\partial D}n_x\cdot \nabla G^{\alpha_0}_2(x-y)\phi(y)d\sigma(y),
		\]
		satisfies
		\begin{equation}\label{eqn:SecondOrderInteg}
			\int_{\partial D_k} \mathcal{K}^{\alpha_0,0}_{D,2}[\phi] d\sigma(y) =-\int_{D_k} \mathcal{S}^{0,0}_{D}[\phi]dx + \iu\frac{|D_k|}{|Y|}\int_{\partial D}(\alpha_0\cdot y)\phi(y)d\sigma(y) ,
		\end{equation}
		for all $ \phi\in \mathbb{L}_0^2(\partial D) $ and $ k = 1,2,\ldots,6 $.
	\end{lemma}
	\begin{proof}
		Note that from \eqref{eqn:JumpSingle} and \eqref{eqn:SingleLayerHatOperator}, we have, for $ \phi\in \mathbb{L}^2_{0}(\partial D) $
		\begin{subequations}
			\begin{equation*}
				\int_{ \partial D_k} -\frac{1}{2}\phi + (\mathcal{K}^{-\omega^2\alpha_0,\omega}_D)^{\ast}[\phi] d\sigma(y) = -\omega^2\int_{D_k} \hat{S}^{\omega^2\alpha_0,\omega}_{D}[\phi]dx   + \mathcal{O}(\omega^4).
			\end{equation*}
		\end{subequations}
		At the same time, we are led to 
		\begin{equation}
			\int_{ \partial D_k}\bigg( -\frac{1}{2}\phi + (\mathcal{K}_D^{-\omega^2\alpha_0,\omega})^{\ast}[\phi] \bigg)d\sigma(y) = \int_{\partial D_k} F^{\alpha_0}[\phi](y)d\sigma(y).
		\end{equation}
		Collecting the above discussions and \eqref{eqn:NormalBdry}, we obtain, for all $ |\omega|\ll 1 $,
		\begin{gather*}
			\begin{aligned}
				\omega^2\int_{ \partial D_k}\mathcal{K}^{\alpha_0,0}_{D,2}[\phi]d\sigma(y)&= -\omega^2\int_{D_k} \hat{S}^{\omega^2\alpha_0,\omega}_{D}[\phi]dx  + \mathcal{O}(\omega^4)\\
				&=-\omega^2\bigg\{\int_{D_k}\mathcal{S}^{0,0}_{D}[\phi]dx -\frac{\iu}{|Y|}\int_{D_k} \bigg[\int_{\partial D} (\alpha_0\cdot y)\phi(y)d\sigma(y)\bigg]dx\bigg\}.
			\end{aligned}
		\end{gather*}
		The $ \mathcal{O}(\omega^2) $ terms are independent of $ \omega $, and they must coincide.
	\end{proof}
	Further, one notes the continuous dependency on $ \delta $ of the original problem. 
	\begin{proposi}
		There are six Bloch resonant frequencies $ \{\omega_{n}(0,\sigma,\delta)\}_{n=1}^6 $ such that $ \{\omega_{n}(0,\sigma,0) = 0\}_{n=1}^6 $ and $ \{\omega_{n}(0,\sigma,\delta)\}_{n=1}^6 $ depends on $ \delta $ continuously. 
	\end{proposi}
	\begin{proof}
		When $ \alpha\neq 0 $, we can adopt the standard procedure to prove that there exist six Bloch resonant frequency $ \{\omega_{n}(\alpha,\sigma,\delta) \}_{n=1}^{6} $ such that $ \{\omega_{n}(\alpha,\sigma,0) = 0 \}_{n=1}^{6}$. See, for example, \cite{Ammari2020}. By the Lipschitz continuity of the band function with respect of $ \alpha $, we claim that $ \{\omega_{n}(0,\sigma,0) = 0\}_{n=1}^{6} $ and there exist at most six Bloch resonant frequencies when $\alpha=0$. \par 
		First, it is clear that $ \omega_1(0,\sigma,0) = 0 $ is an eigenvalue when $\alpha=0$ for all $\delta\ge0$. So we consider the eigenvalues that correspond to the non-constant eigenfunctions. When $\delta=0$, we investigate the operator $ \mathcal{A}^{0,\omega}_{0} $ acting on $ (\phi,\psi)^{T} $ and decompose it into two parts:
		\begin{gather}\label{eqn:OperatorDecompose}
			\begin{aligned}
				\mathcal{A}^{0 ,\omega}_{0}\begin{pmatrix}
					\phi\\ \psi	
				\end{pmatrix}
				&= \begin{pmatrix}
					\mathcal{S}^{0,k_1}_D[\phi] - \frac{1}{k_1^2|Y|}\int_{\partial D} \phi d\sigma - \mathcal{S}^{0,k_0}_D[\psi]+\frac{1}{k_0^2|Y|}\int_{\partial D}\psi d\sigma\\
					-\frac{1}{2}\phi +(\mathcal{K}^{0,k_1}_D)^{\ast}[\phi] 
				\end{pmatrix} \\&\quad + 
				\begin{pmatrix}
					\frac{1}{k_1^2|Y|}\int_{\partial D} \phi d\sigma -\frac{1}{k_0^2|Y|}\int_{\partial D}\psi d\sigma\\
					0
				\end{pmatrix}.
			\end{aligned}
		\end{gather}
		The first part is a holomorphic operator-valued function in a neighborhood of $ \omega = 0 $, while the second part is finitely meromorphic\cite{Ammari2018}. \par 
		From \eqref{eqn:OperatorDecompose}, the point $ \omega = 0 $ is a characteristic value of $ \mathcal{A}^{0,\omega}_{0}  $. To see this we notice that 
		\[ \left\{\Phi_j = \begin{pmatrix}
			\varphi_j\\
			\varphi_j
		\end{pmatrix}\right\}_{j=1}^{5}, \]
		are the root functions of $ \mathcal{A}^{0,\omega}_{0} $ associated to $\omega=0$. Here $ \{\varphi_j\}_{j=1}^5 $ forms the basis of $\ker(-\frac{1}{2}\operatorname{Id} + (\mathcal{K}^{0,0}_{D})^{\ast})\subset \mathbb{L}^{2}_{0}(\partial D) $. And by \eqref{eqn:AsymExpZero}, we have 
		\begin{equation}
			-\frac{\varphi_j}{2} + (\mathcal{K}^{0,k_1}_{D})^{\ast}[\varphi_j] = \omega^2 h(x,\omega),x\in \partial D,\quad j = 1,2,\ldots,5,
		\end{equation}
		for some function $ h $ that is holomorphic with respect to $ \omega $ in a neighborhood of $ \omega = 0 $. Since 
		\[\int_{\partial D_j}\mathcal{S}^{0,0}_D[\varphi_j]dx\neq 0,\quad j = 1,2,\ldots,5,\]
		the rank of each $ \{\Phi_j\}_{j=1}^{5} $ is two. Thus, the multiplicity of the characteristic value of $\omega=0$ is $ 10 $. \par 
		The Neumann-Poincar\'e operator $ (\mathcal{K}^{0,0}_D)^{\ast} $ is known to be a compact operator, so $ -\frac{1}{2}\operatorname{Id} + (\mathcal{K}^{0,0}_D)^{\ast} $ is Fredholm of index $ 0 $. Since $ \mathcal{S}^{0,0}_D $ is also Fredholm, we conclude that $ \mathcal{A}^{0,0}_0 $ is of Fredholm type. Since $ (\mathcal{K}^{0,\omega}_{D})^{\ast} $ is holomorphic and $ \mathcal{S}^{0,\omega}_D $ is finitely meromorphic as a function of $\omega$ is a neighborhood of $ 0 $, it follows that $ \mathcal{A}^{0,\omega}_0 $ is of Fredholm type. \par 
		Choosing a disk $V\subset \mathbb{C}$ around $0$, with small enough radius such that $\mathcal{A}^{0,\omega}_0$ is invertible on $\partial V$ and $\omega =0$ is the only characteristic value in $V$. From Lemma 1.11 in \cite{Ammari2018}, it follows that $\mathcal{A}^{0,\omega}_0$ is normal with respect to $\partial V$. \par  
		
		Now we consider the full operator $ \mathcal{A}^{0,\omega}_{\delta} $. It is obvious that 
		\begin{gather*}
			\begin{aligned}
				\mathcal{A}^{0,\omega}_{\delta} &= \mathcal{A}^{0,\omega}_{0} + \begin{pmatrix}
					0 & 0\\
					0 & -\delta\left( \frac{1}{2}\operatorname{Id} + (\mathcal{K}^{0,\omega}_{D})^{\ast} \right)
				\end{pmatrix}\\
				&\triangleq \mathcal{A}^{0,\omega}_0 + A^{(1)}(\omega,\delta),
			\end{aligned}
		\end{gather*}
		where $  A^{(1)}(\omega,\delta) $, is holomorphic with respect to $ \omega $ in $ V $ and continuous up to $ \partial V $. For small enough $ \delta $ we have 
		\[
		\left\Vert (\mathcal{A}^{0,\omega}_0)^{-1}A^{(1)}(\omega,\delta) \right\Vert_{\mathcal{B}((\mathbb{L}^{2}(\partial D)^2),(\mathbb{L}^{2}(\partial D)^2))}<1,\quad \omega \in \partial V. 
		\]
		Hence by the generalization of Rouch\'e's theorem, the multiplicity of the operator $\mathcal{A}^{0,\omega}_{\delta}$ inside $ V $ equals to 10, for sufficient small $ \delta $. Since the characteristic values are symmetric around the origin, it is clear that there exist 5 characteristic values with positive real part. Adding one due to the constant function, we conclude that there exits 6 Bloch resonant frequency. 
	\end{proof}
	Now we can give a complete proof to \Cref{thm:EigValAsympt}.
	\begin{proof}[Proof of \Cref{thm:EigValAsympt}]
		After rescaling the quasi-periodic vector $\alpha = k_1^2\alpha_0$, One look for normalized solutions $ (\phi^{\alpha},\psi^{\alpha})^{T} $ to the integral equation.
		\[ \mathcal{A}^{k_1^2\alpha_0,\omega}\begin{pmatrix}
			\phi^{\alpha}\\ \psi^{\alpha}
		\end{pmatrix} = \begin{pmatrix}
			\mathcal{S}^{k_1^2\alpha_0,k_1}_{D}[\phi^{\alpha}]    -\mathcal{S}^{k_0^2\tilde{\alpha_0},k_0}_{D}[\psi^{\alpha}]\\
			-\frac{1}{2}\phi^{\alpha} + \left(\mathcal{K}^{-k_1^2\alpha_0,k_1}_{D}\right)^{\ast}[\phi^{\alpha}]  -\delta\left[\frac{1}{2}\psi^{\alpha} + \left(\mathcal{K}^{-k_0^2\tilde{\alpha_0},k_0}_{D}\right)^{\ast}[\psi^{\alpha}]\right] \end{pmatrix} = 0, \]
		where $ \tilde{\alpha_0} = \frac{v_0^2}{v_1^2}\alpha_0  $.   
		By the asymptotic expansion, we are led to 
		\begin{align}
			\hat{S}^{k_1^2\alpha_0,k_1}_{D}[\phi^{\alpha}]- \hat{S}^{k_0^2\tilde{\alpha_0},k_0}_D[\psi^{\alpha}] = \mathcal{O}(\omega^2 ),\label{eqn:ContinuousValue}\\
			F^{\alpha_0}[\phi^{\alpha}] +k_1^2\mathcal{K}^{\alpha,0}_{D,2}[\phi^{\alpha}] -\delta(\psi^{\alpha}+ F^{\tilde{\alpha_0}}[\psi^{\alpha}])&= \mathcal{O}(\omega^4 + \omega^2\delta)\label{eqn:JumpDerivative}.
		\end{align}
		Since the operators $ \hat{S}^{k_1^2\alpha_0,k_1}_{D} $ and $ \hat{S}^{k_0^2\alpha_0,k_0}_D $ are invertible uniformly when \newline $ |\alpha|<\min(k_1^2,k_0^2) $ and $ \omega \neq 0 $, one has
		\[ \psi^{\alpha} = \phi^{\alpha} + \mathcal{O}(\omega^2), \quad |\alpha|<\min(k_1^2,k_0^2).\]
		Inserting the above expression into \eqref{eqn:JumpDerivative}, we obtain 
		\[ F^{\alpha_0}[\phi^{\alpha}] = \mathcal{O}(\omega^2 + \delta). \] 
		Thus from \Cref{lem:FOperator}, the solution $(\phi^{\alpha},\psi^{\alpha})$ to the integral equations can be rewritten as
		\begin{equation}\label{eqn:AsymExpansion}
			\phi^{\alpha} = \varphi_0 + \phi_{R}^{\alpha_0},\quad \psi^{\alpha} = \varphi_0 + \psi_R^{\alpha_0}.
		\end{equation}
		Here $\varphi_0\in \ker(-\frac{1}{2}\operatorname{Id} + (\mathcal{K}^{0,0}_{D})^{\ast}) $, while $ \phi^{\alpha}_R , \psi_{R}^{\alpha_0}$ lie in the orthogonal complement of \newline $ \ker(-\frac{1}{2}\operatorname{Id} + (\mathcal{K}^{0,0}_{D})^{\ast}) $ in $ \mathbb{L}^2(\partial D)$ and is of $ \mathcal{O}(\omega^2) $.  \par 
		The asymptotic expansion \eqref{eqn:AsymExpansion} yield
		\[ F^{\alpha_0}[\phi^{\alpha}_{R}] = \mathcal{O}(\omega^2+\delta).\]
		When restricted to the orthogonal complement of $ \ker(F^{\alpha_0}) $, the operator $ F^{\alpha_0} $ is invertible, then it follows that  $\phi^{\alpha}_{R}= \mathcal{O}(\omega^2+\delta) $ uniformly for $|\alpha| < \min(k_1^2,k^2)$.\par 
		Now from \eqref{eqn:ContinuousValue}, we have 
		\[  \mathcal{S}^{0,0}_{D}[\varphi_0] - \frac{\iu}{|Y|}\int_{\partial D}(\alpha_0\cdot y)\varphi_0d\sigma(y) + \frac{1}{k_1^2|Y|}\int_{\partial D}\phi_R^{\alpha_0} d\sigma(y)= a_j\chi_{_{\partial D_j}} + \mathcal{O}(\omega^2).\]
		Here and thereafter in the proof, the Einstein summation convention will be adopted, $ a_j = \sum_{j=1}^6 a_j $. From \Cref{proposi:asymptBehav}, we conclude that 
		\begin{equation}\label{eqn:FirstOrderSolution}
			\textstyle\phi^{\alpha} = a_j\phi_{j,0}+\phi_{R}^{\alpha_0}.
		\end{equation}
		Substituting \eqref{eqn:FirstOrderSolution} into \eqref{eqn:JumpDerivative} and integrating around $\partial D_k $ for $k=1,2,\ldots,6$, one has
		\[
		\int_{\partial D_k}\bigg\{k_1^2\mathcal{K}^{\alpha_0,0}_{D,2}[a_j\phi_{j,0}] + F^{\alpha_0}[\phi_r^{\alpha_0}] \bigg\}d\sigma(y)  + \delta  a_j\int_{ \partial D_k}\phi_{j,0}d\sigma(y) = \mathcal{O}(\omega^4+\delta\omega^2).
		\]
		Combining statement (i) of \Cref{proposi:asymptBehav} and \eqref{eqn:SecondOrderInteg}, we have
		\[
		k_1^2\int_{\partial D_k}\mathcal{K}^{\alpha_0,0}_{D,1}[a_j\phi_{j,0}] d\sigma(y) = -k_1^2|D_k|(a_k - \frac{1}{6}a_j) + \iu k_1^2\frac{|D_k|}{|Y|}\int(\alpha_0\cdot y)a_j\phi_{j,0}d\sigma(y).
		\]
		Then from the definition of the operator $ F^{\alpha_0} $, one notices that 
		\begin{equation*}
			\int_{\partial D_k}F^{\alpha_0}[\phi_R^{\alpha_0}] d\sigma(y)= \int_{ \partial D} \phi_R^{\alpha_0}(-\frac{1}{2}\operatorname{Id} + \mathcal{K}^{0,0}_D)[\chi_{_{\partial D_k}}]d\sigma(y).
		\end{equation*}
		Invoking \Cref{aplem:KernelKOp}, we are led to 
		\[
		(-\frac{1}{2}\operatorname{Id} + \mathcal{K}^{0,0}_D)[\chi_{_{\partial D_k}}-1/6\chi_{_{\partial D}}+1/6\chi_{_{\partial D}}] = (-\frac{1}{2}\operatorname{Id} + \mathcal{K}^{0,0}_D)[1/6\chi_{_{\partial D}}] = -\frac{|D|}{6|Y|}.
		\]
		And we have 
		\begin{equation}
			-\frac{\omega^2|D|}{6v_1^2}a_k + \delta C^{0,\sigma}_{jk}a_k= \mathcal{O}(\omega^4+\delta \omega^2),\quad j = 1,2,\ldots,6.
		\end{equation}
		Therefore, the value $\frac{\omega^2|D|}{6\delta v_1^2}$ approximate the eigenvalue of the periodic capacitance matrix $\mathbf{C}^{0}$. This complete the proof of \Cref{thm:EigValAsympt}. 
	\end{proof}
	
	\section{Structure of Periodic Capacitance Matrix}\label{sec:structure}
	Now that we have successfully associated the subwavelength frequencies with the eigenvalues of the periodic capacitance matrix, we will investigate its structure in this section. And we will prove the following result, which indicate the existence of local band gap for $ \sigma\neq 0 $:
	\begin{proposi}\label{proposi:capacitanceMatEig}
		The eigenvalues of the periodic capacitance matrix $\mathbf{C}^0$ satisfy 
		\[ 0=\lambda_1< \lambda_2=\lambda_3 \le \lambda_4 = \lambda_5<\lambda_6. \]
		Moreover, when $ \sigma = 0 $, $ \lambda_3=\lambda_4 $. When $ \sigma \neq 0 $, $ \lambda_3<\lambda_4 $.
	\end{proposi}
	Before we give the full proof, we first give a proposition describing its structure.
	\begin{proposi}
		The periodic capacitance matrix $ \mathbf{C}^{0} $ has the following structures 
		\begin{itemize}
			\item[(i)] For all $j,k = 1,2,\ldots, 6$, $ C^{0}_{jk}= C^{0}_{j-k} $ and $ C^{0}_{jk} = C^{0}_{kj} $.
			\item[(ii)] $ C^{0}_{13} = C^{0}_{15} $, $ C^{0}_{12} = C^{0}_{16} $. 
			\item[(iii)] For all $j,k = 1,2,\ldots, 6$, $j\neq k,$ 
			\begin{equation*}
				\sum_{j=1}^{6}C_{jk}^{0}=0,\quad C^{0}_{jj}>0,\quad C^{0}_{jk}<0.
			\end{equation*}
		\end{itemize}
	\end{proposi}
	\begin{proof}
		To prove ($ i $), for any $ C^{0}_{jk} $we calculate directly, 
		\[ \int_{\partial D_{k}}\phi_{j,0}(y)d\sigma(y)  = \int_{\partial D_{k}}\phi_{j+1,0}(Ry)d\sigma(y) = \int_{\partial D_{k+1}}\phi_{j+1,0}(y)d\sigma(y)=C^0_{j+1,k+1}.  \]
		And the equality $ C^{0}_{16} = C^{0}_{21}, C^{0}_{61} = C^{0}_{12}$ can be proved in a similar way. The symmetry can be proved by invoking \Cref{rmk:RepresentationCapacity}. \par 
		As for ($ ii $), the proof goes similarly as ($ i $). Take the equality $ C^{0}_{12} = C^{0}_{16} $ for example:
		\[ \int_{\partial D_{2}}\phi_{1,0}(y)d\sigma(y)  = \int_{\partial D_{2}}\phi_{2,0}(T_{y}y)d\sigma(y) = \int_{\partial D_{1}}\phi_{2,0}(y)d\sigma(y)=C^0_{21}=C^{0}_{16}.  \]\par  
		Finally we will prove ($ iii $). This follows from the facts that $\phi_{j,0}\in \mathbb{L}^2_0(\partial D)$ and the single layer potential $\mathcal{S}^{0,0}_{D}[\phi_{j,0}]$ is harmonic in $Y\backslash D$.  
		We have 
		\[ \sum_{j=1}^{6}C_{jk}^{0}=\int_{\partial D}\phi_{k,0}d\sigma(y) = 0. \]
		And by the maximum principle, the normal derivative of the single layer potential should be strictly negative on $\partial D_j$, and be strictly positive on $ \partial D\backslash \partial D_j $. 
	\end{proof}\par
	To fully describe the eigenvalues of the periodic capacitance matrix, we need to characterize each element quantitatively. In \cite{Ammari2019,Ammari2020b}, the authors are able to investigate the capacitance matrices for finite systems or a periodic chain in three-dimension quantitatively.
	\begin{lemma}\label{proposi:ElementChange}
		If $ \operatorname{dist}(P_i+\Lambda,P_j) > \operatorname{dist}(P_i+\Lambda,P_{k}) $ for $ i\neq j $ and $ i\neq k $, where $ P_i+\Lambda \triangleq \{ P_i+l:l\in \Lambda \} $. We have 
		\[ C^0_{ik}<C^{0}_{ij}. \]
	\end{lemma}
	From this fact we conclude that $ C^{0}_{12} <C^{0}_{13}, C^{0}_{14}<C^{0}_{13} $. Moreover, when $ \sigma=0 $, we can prove by symmetry that $ C^{0}_{12} = C^{0}_{14} $
	And we can prove \Cref{proposi:capacitanceMatEig} from the basic linear algebra. Together with \Cref{thm:EigValAsympt}, and the continuity of subwavelength frequency with respect to $ \alpha $,  we can conclude that there exist a local band gap near $ \alpha = 0 $ when $ \sigma \neq 0 $. 
	\begin{theorem}\label{thm:LocalBandGap}
		For contracted or dilated generalized honeycomb-structured inclusions, there exists an open set $ V\subset Y' $ containing $ \alpha = 0 $, such that for all $ \alpha\in V $,
		\[ \omega_{3}^2(\alpha,\sigma,\delta)<\omega_{4}^2(\alpha,\sigma,\delta). \]
	\end{theorem}\par 
	We can also easily derive the corresponding eigenvectors.
	\begin{proposi}\label{proposi:EigFuncPer}
		Let $ \{v_j\}_{j=1}^6 $ be the six eigenvectors of the periodic capacitance matrix $ \mathbf{C}^0 $. When $\sigma<0$, the second to fifth eigenvectors of the periodic capacitance matrix are given by 
		\begin{gather*}
			\begin{aligned}
				v_2 = (1,1,0,-1,-1,0)^{T},&\quad v_3 = (1,2,1,-1,-2,-1),\\
				v_4 = (1,-1,0,1,-1,0)^{T},&\quad v_5 = (1, -2, 1, 1 , -2,1 ).
			\end{aligned}
		\end{gather*}
		When $\sigma>0$, the second to fifth eigenvectors of the periodic capacitance matrix are given by 
		\begin{gather*}
			\begin{aligned}
				v_2 = (1,-1,0,1,-1,0)^{T},&\quad v_3 = (1, -2, 1, 1 , -2,1 ),\\
				v_4 = (1,1,0,-1,-1,0)^{T},&\quad v_5 = (1,2,1,-1,-2,-1).
			\end{aligned}
		\end{gather*}
	\end{proposi}
	From \Cref{thm:EigValAsympt}, the eigenfunctions of \eqref{eqn:ProblemFormulation} can be asymptotically determined up to $ \mathcal{O}(\delta) $ 
	\[ \textstyle w_l \sim \sum_{j=1}^6\mathcal{S}^{0,0}_{D}[\phi_{j,0}](v_l)_{j}, \quad l = 2,3,4,5. \]
	Here recall that $ \phi_{j,0} $ are determine by \eqref{eqn:expansion}. The $j$-th component of vector $v_l$ is denoted by $(v_l)_j$. 
	\begin{remark}
		We comment here that the different symmetries of eigenfunctions of \eqref{eqn:ProblemFormulation} may carry topological information. It also shed light on the existence of topological states on the edge. This will be left to future works.
	\end{remark}
	\section{Band Structure of Super Honeycomb-structured Materials}\label{sec:ConeStructure}
	In this section, we prove the additional symmetry of super honeycomb-structured inclusions when $ \sigma = 0 $. Following this, we will establish the band folding result \Cref{thm:Decomposition} to justify the existence of double Dirac point structure near $ \alpha=0 $. 
	\subsection{Double Dirac Cone Structure}
	When $\sigma = 0$, we can prove easily that the inclusions $ D_{per} $ has the following additional invariance.
	\begin{proposi}\label{proposi:subPeriodic}
		The super honeycomb-structured inclusions $ D_{per} $ is invariant under the translation $ \tilde{l}_1 $ and $ \tilde{l}_2 $.   
	\end{proposi}
	\par 
	Thus we have shown that for super honeycomb-structured inclusions, we can choose a smaller unit cell $ \tilde{Y} $ by \Cref{proposi:subPeriodic}, as plotted in \Cref{fig:Lattice}. It is enclosed in blue dashed lines and has two distinct inclusions $ D_1 $ and $ D_2 $, and we let $ \tilde{D} \triangleq D_1\cup D_2 $. Similarly we can define the lattice points $ \tilde{\Lambda} \triangleq \{ m_1\tilde{l}_1+m_2\tilde{l}_2:m_1,m_2\in \mathbb{Z} \} $ and the corresponding dual lattice points $ \tilde{\Lambda}^{\prime} \triangleq \{ m_1\tilde{k}_1+m_2\tilde{k}_2:m_1,m_2\in \mathbb{Z} \} $, where 
	\begin{equation}\label{eqn:BigDualBasis}
		\tilde{k}_1= 2k_1-k_2 ,\quad  \tilde{k}_2= 2k_2-k_1.
	\end{equation}
	As stated in \Cref{proposi:subPeriodic}, we can therefore define the Green's functions $ \tilde{G}^{\alpha,\omega} $ and $ \tilde{G}^{0,0} $ as defined in \eqref{eqn:QuasiGreensFunc} and \eqref{eqn:PeriodicGreenFunc}. 
	\[ \tilde{G}^{\alpha ,\omega} = \frac{3}{|Y|}\sum_{q\in \tilde{\Lambda}'}\frac{\eu^{\iu(\alpha+q)\cdot x}}{\omega^2-|\alpha+q|^2},\quad \tilde{G}^{0 ,0} = -\frac{3}{|Y|}\sum_{q\in \tilde{\Lambda}'\backslash\{0\}}\frac{\eu^{\iu q\cdot x}}{|q|^2}.  \]
	Therefore, single layer potentials can be defined similarly. And further we can prove the band folding theorem: 
	\begin{theorem}\label{thm:Decomposition}
		For sufficient small $ |\omega|>|\alpha|\ge 0 $, $|\varepsilon|\ge0$, or $ \omega = |\alpha| = |\varepsilon| $, the following identities holds for arbitrary $\phi\in \mathbb{L}^2(\partial D)$
		\begin{align}
			\mathcal{S}_D^{\varepsilon,\omega}[\phi] &= \tilde{\mathcal{S}}^{\alpha_1+\varepsilon,\omega}_{\tilde{D}}[\tilde{\phi}^{\varepsilon}_1]+ \tilde{\mathcal{S}}^{\alpha_2+\varepsilon,\omega}_{\tilde{D}}[\tilde{\phi}^{\varepsilon}_1] + \tilde{\mathcal{S}}^{\alpha_{3}+\varepsilon,\omega}_{\tilde{D}}[\tilde{\phi}^{\varepsilon}_3],\label{eqn:SingleLayerDecomp}\\
			(\mathcal{K}_D^{-\varepsilon,\omega})^{\ast}[\phi] &= (\tilde{\mathcal{K}}^{-\alpha_1-\varepsilon,\omega}_{\tilde{D}})^{\ast}[\tilde{\phi}^{\varepsilon}_1]+ (\tilde{\mathcal{K}}^{-\alpha_2-\varepsilon,\omega}_{\tilde{D}})^{\ast}[\tilde{\phi}^{\varepsilon}_1] + (\tilde{\mathcal{K}}^{-\alpha_{3}-\varepsilon,\omega}_{\tilde{D}})^{\ast}[\tilde{\phi}^{\varepsilon}_3].
		\end{align}
		Here $ \tilde{\phi}_{j}^{\varepsilon}\in \mathbb{L}^2_{\alpha_{j}+\varepsilon}(\partial \tilde{D}) $, $ j = 1,2,3 $ and can be determined uniquely, as shown in \eqref{eqn:Restriction1}. The single layer potential and the Neumann-Poincar\'e operator are defined by 
		\begin{gather*}
			\begin{aligned}
				\tilde{\mathcal{S}}^{\alpha_{j}+\varepsilon,\omega}_{\tilde{D}}[\tilde{\phi}^{\varepsilon}_j]&\triangleq \int_{ \partial \tilde{D}} \tilde{G}^{\alpha_{j}+\varepsilon}(x-y)\tilde{\phi}^{\varepsilon}_{j}(y)d\sigma(y),\\
				(\tilde{\mathcal{K}}^{-\alpha_{j}-\varepsilon,\omega}_{\tilde{D}})^{\ast}[\tilde{\phi}^{\varepsilon}_j]&\triangleq 
				\operatorname{p.v.}\int_{ \partial \tilde{D}}n_x\cdot \nabla \tilde{G}^{\alpha_{j}+\varepsilon}(x-y)\tilde{\phi}^{\varepsilon}_{j}(y)d\sigma(y),
			\end{aligned}
		\end{gather*}
		The vectors $ \alpha_{1},\alpha_{2},\alpha_{3}  $ are defined by 
		\begin{equation}\label{eqn:HighSymmetry}
			 \alpha_{1} \triangleq \frac{2}{3}\tilde{k}_1 + \frac{1}{3}\tilde{k}_2 = k_1,\quad \alpha_2 \triangleq \frac{1}{3}\tilde{k}_1 + \frac{2}{3}\tilde{k}_2 = k_2,\quad \alpha_{3}=0.
		\end{equation}
	\end{theorem}
	\begin{remark}
		In other literature, the points $ \alpha_1,\alpha_{2} $ are often referred to as high symmetry points or diabolical points, see \cite{Ablowitz2009}. And for simplicity we omit the $ \alpha_{3} $ in the following proof.       
	\end{remark}
	To prove this theorem, we first give several lemmas:
	\begin{lemma}\label{lem:Decompose}
		The following decomposition holds
		\begin{equation}
			\Lambda' =   (\alpha_1 + \tilde{\Lambda}')  \cup (\alpha_2 + \tilde{\Lambda}')\cup \tilde{\Lambda}'. 
		\end{equation}
	\end{lemma}
	
	\begin{proof}
		First we rewrite the dual lattice point $ \tilde{\Lambda}' $ by $ \tilde{\Lambda}' = \{m_1\tilde{k}_1 +m_2\tilde{k}_2: m_1,m_2\in \mathbb{Z}\}. $
		From \eqref{eqn:BigDualBasis}, we have 
		\[
		\tilde{\Lambda}' = \{(2m_1-m_2)k_1 +(2m_2-m_1)k_2: m_1,m_2\in \mathbb{Z}\}.
		\]
		Similarly we have 
		\begin{gather*}
			\begin{aligned}
				\alpha_1 + \tilde{\Lambda}' &= \{ (2m_1-m_2+1)k_1 +(2m_2-m_1)k_2: m_1,m_2\in \mathbb{Z} \}, \\
				\alpha_{2} + \tilde{\Lambda}' &= \{ (2m_1-m_2)k_1 +(2m_2-m_1+1)k_2: m_1,m_2\in \mathbb{Z} \}. 
			\end{aligned}
		\end{gather*}
		By the simple fact that
		\begin{equation}
			\mathbb{Z}^2 = \{ (2m_1-m_2,2m_2-m_1) \}\cup \{ (2m_1-m_2+1,2m_2-m_1) \} \cup \{ (2m_1-m_2,2m_2-m_1+1) \},
		\end{equation}
		we can conclude that
		\[   (\alpha_1 + \tilde{\Lambda}')  \cup (\alpha_2 + \tilde{\Lambda}')\cup\tilde{\Lambda}' = \{ n_1k_1+n_2k_2:n_1,n_2\in \mathbb{Z} \}= \Lambda' .\] 
	\end{proof}
	\begin{remark}
		By altering the above procedures slightly we can also prove that
		\[ \Lambda'\backslash\{0\} =   (\alpha_1 + \tilde{\Lambda}')  \cup (\alpha_2 + \tilde{\Lambda}') \cup(\tilde{\Lambda}'\backslash\{0\}). \]
	\end{remark}\par 
	From \Cref{lem:Decompose}, we immediately draw the conclusion that the single layer potential can be written as follows: 
	\begin{equation}
		\mathcal{S}_D^{\varepsilon,\omega}[\phi] = \frac{1}{3}\int_{\partial D}[\tilde{G}^{\alpha_1+\varepsilon,\omega}+\tilde{G}^{\alpha_2+\varepsilon,\omega}+\tilde{G}^{\varepsilon,\omega}(x-y)]\phi(y)d\sigma(y) .
	\end{equation}
	To simplify the above formula further, we decompose the space $ \mathbb{L}^2(\partial D) $ into three subspaces that are pairwise orthogonal. 
	\begin{lemma}\label{proposi:Decomposition}
		For sufficiently small $ |\varepsilon| $, the space $\mathbb{L}^2(\partial D)$ has an orthogonal decomposition as follows 
		\begin{equation}
			\mathbb{L}^2(\partial D) =    \mathbb{L}_{\alpha_1+\varepsilon}^2(\partial D)\oplus \mathbb{L}_{\alpha_2+\varepsilon}^2(\partial D)\oplus\mathbb{L}_{\varepsilon}^{2}(\partial D),
		\end{equation}
		where $ \mathbb{L}^{2}_{\tilde{\alpha}}(\partial D) $ for $ \tilde{\alpha}\in \tilde{Y}^{\prime} $ denotes the $ \tilde{\alpha} $ quasi-periodic extension of $ \mathbb{L}^{2}(\partial \tilde{D}) $ functions. 
		\begin{gather}
			\begin{aligned}
				\mathbb{L}_{\tilde{\alpha}}^{2}(\partial D) &\triangleq \{ f: f  = \chi_{_{\partial D_1}}f_1(y) + \eu^{\iu\tilde{\alpha}\cdot l_1}\chi_{_{\partial D_3}}f_1(y-l_1)+\eu^{\iu\tilde{\alpha}\cdot(l_1-l_2)}\chi_{_{\partial D_5}}f_1(y-l_1+l_2)\\
				&\qquad +\chi_{_{\partial D_2}}f_2(y)+\eu^{\iu\tilde{\alpha}\cdot(l_1-l_2)}\chi_{_{\partial D_4}}f_2(y-l_1+l_2)+\eu^{-\iu\tilde{\alpha}\cdot l_2}\chi_{_{\partial D_6}}f_2(y+l_2)\}.
			\end{aligned}
		\end{gather}
		Here $f_1\in \mathbb{L}^2(\partial D_1)$, $f_2\in \mathbb{L}^2(\partial D_2)$ and $ \partial \tilde{D} = \partial D_1\cup \partial D_2 $.
	\end{lemma}
	\begin{proof}
		For simplicity we only prove \eqref{eqn:SingleLayerDecomp}. We first prove that for any $g\in \mathbb{L}^2(\partial D) $ given by $ g = \sum_{l=1}^{6}g_l\chi_{_{\partial D_l}}, $
		we can uniquely solve the following linear equations 
		\begin{align} 
			\begin{pmatrix}
				1 & \eu^{\iu\varepsilon\cdot l_1} & \eu^{\iu\varepsilon\cdot(l_1-l_2)}\\
				1 & \eu^{\iu\frac{4\pi}{3}+\iu\varepsilon\cdot l_1} & \eu^{\iu\frac{2\pi}{3}+\iu\varepsilon\cdot(l_1-l_2)}\\
				1 & \eu^{\iu\frac{2\pi}{3}+\iu\varepsilon\cdot l_1} & \eu^{\iu\frac{4\pi}{3}+\iu\varepsilon\cdot(l_1-l_2)}
			\end{pmatrix}
			\begin{pmatrix}
				f_{1,\varepsilon} \\
				f_{1,\alpha_1+\varepsilon}\\
				f_{1,\alpha_2+\varepsilon}
			\end{pmatrix} = 
			\begin{pmatrix}
				g_1\\
				g_3\\
				g_5
			\end{pmatrix},\\
			\begin{pmatrix}
				1 & \eu^{\iu\varepsilon\cdot(l_1-l_2)} & \eu^{-\iu\varepsilon\cdot l_2}\\
				1 & \eu^{\iu\frac{2\pi}{3}+\iu\varepsilon\cdot(l_1-l_2)} & \eu^{\iu\frac{4\pi}{3}-\iu\varepsilon\cdot l_2}\\
				1 & \eu^{\iu\frac{4\pi}{3}+\iu\varepsilon\cdot(l_1-l_2)} & \eu^{\iu\frac{2\pi}{3}-\iu\varepsilon\cdot l_2}
			\end{pmatrix}
			\begin{pmatrix}
				f_{2,\varepsilon} \\
				f_{2,\alpha_1+\varepsilon}\\
				f_{2,\alpha_2+\varepsilon}
			\end{pmatrix} = 
			\begin{pmatrix}
				g_2\\
				g_4\\
				g_6
			\end{pmatrix}.
		\end{align}
		The orthogonality can be easily deduced from the equality $ 1+\eu^{2\pi\iu/3} + \eu^{4\pi\iu/3} =0$.
	\end{proof}\par 
	\begin{proof}[Proof of \Cref{thm:Decomposition}]
	Given \Cref{proposi:Decomposition}, we can decompose the function $ \phi\in \mathbb{L}^2(\partial D)$ as 
	\begin{equation}\label{eqn:Restriction1}
		\phi = \sum_{k=1}^3\phi^{\varepsilon}_k,\quad \phi_1^{\varepsilon}\in \mathbb{L}_{\alpha_1+\varepsilon}^2(\partial D),\; \phi_2^{\varepsilon}\in \mathbb{L}_{\alpha_2+\varepsilon}^2(\partial D),\;\phi_3^{\varepsilon}\in \mathbb{L}_{\varepsilon}^{2}(\partial D).
	\end{equation}
	Supposing $ x\in \partial \tilde{D} $, we first calculate:
	\[
	\frac{1}{3}\int_{\partial D} \tilde{G}^{\varepsilon,\omega}(x-y)\phi(y)d\sigma(y)  = \sum_{l=1}^6\frac{1}{3}\int_{ \partial D_l}\tilde{G}^{\varepsilon,\omega}(x-y)\phi(y)d\sigma(y).
	\] 
	Since we have
	\begin{equation*}
		\int_{ \partial D_3}\tilde{G}^{\varepsilon,\omega}(x-y)\phi(y)d\sigma(y) 
		= \int_{ \partial D_1} \tilde{G}^{\varepsilon,\omega}(x-y)[\phi_{1}^{\varepsilon}(y)+\eu^{4\pi\iu/3}\phi_{2}^{\varepsilon}(y) + \eu^{2\pi\iu/3}\phi_{3}^{\varepsilon}(y)]d\sigma(y),
	\end{equation*}
	we can deduce in the same manner for each inclusion $ \{\partial D_l\}_{l=1}^6 $, obtaining
	\[
	\int_{\partial \tilde{D}}\tilde{G}^{\varepsilon,\omega}(x-y)\phi_1^{\varepsilon}(y)d\sigma(y)=\tilde{\mathcal{S}}^{\varepsilon,\omega}_{\tilde{D}}[\tilde{\phi}_1^{\varepsilon}](x),
	\]
	again by the equality $ 1+\eu^{2\pi\iu/3}+\eu^{4\pi\iu/3}=0 $. Here $ \tilde{\phi}_{1}^{\varepsilon} $ denote the function $ \phi^{\varepsilon}_{1} $ restricted on $ \partial \tilde{D} $. 
	Similar procedures yield:
	\begin{subequations}
		\begin{align*}
			\frac{1}{3}\int_{\partial D}\tilde{G}^{\alpha_1+\varepsilon,\omega}(x-y)\phi(y)d\sigma(y) &= \tilde{\mathcal{S}}_{\tilde{D}}^{\alpha_1+\varepsilon,\omega}[\tilde{\phi}_2^{\varepsilon}](x),\\
			\frac{1}{3}\int_{\partial D} \tilde{G}^{\alpha_2+\varepsilon,\omega}(x-y)\phi(y)d\sigma(y) &= \tilde{\mathcal{S}}^{\alpha_2+\varepsilon,\omega}_{\tilde{D}}[\tilde{\phi}_3^{\varepsilon}](x) ,
		\end{align*}
	\end{subequations}
	where $ \phi^{\varepsilon}_{2}, \phi^{\varepsilon}_{3} $ restricted on $ \partial \tilde{D} $. To determine the value of $ \tilde{\mathcal{S}}^{\varepsilon,\omega}_{\tilde{D}}[\tilde{\phi}_1^{\varepsilon}](x), \tilde{\mathcal{S}}^{\alpha_1+\varepsilon,\omega}_{\tilde{D}}[\tilde{\phi}_2^{\varepsilon}](x),$ and $\tilde{\mathcal{S}}^{\alpha_2+\varepsilon,\omega}_{\tilde{D}}[\tilde{\phi}_3^{\varepsilon}](x)$ for $ x\in \partial D\backslash \partial \tilde{D} $, we extend them by quasi-periodicity $ \varepsilon,\alpha_{1}+\varepsilon, \alpha_{2}+\varepsilon $ correspondingly. \par 
	By altering the overall procedures a little, one can prove the decomposition when $ \omega=0 $, so we have proved \Cref{thm:Decomposition}.
	\end{proof}
	
	In \Cref{thm:Decomposition}, we have established that if one wishes to describe the band structure $ \{ \omega_{n}(\alpha,\sigma,\delta)\}_{n\in \mathbb{N}} $ near $ \alpha = 0 $, we only have to describe the band structure of similar problem \eqref{eqn:ProblemFormulation} defined on $ \tilde{D} $ near $ \alpha_1 ,\alpha_{2}$, and $\alpha_{3} = 0 $. \par 
	Recalling the results in \Cref{subsec:SuperHoneySymmetry} and \Cref{proposi:subPeriodic}, it follows directly that the inclusions inside the cell $ \tilde{D} $ satisfy the symmetry condition posed in \cite{Ammari2020}. Therefore we can invoke the result:
	\begin{theorem}[Theorem 4.1, \cite{Ammari2020}]\label{thm:DiracConeSmall}
		For sufficiently small $ \delta $, the first and second characteristic values of the operator-valued function 
		\[ \tilde{\mathcal{A}}^{\alpha,\omega}_{\delta}\triangleq \begin{pmatrix}
			\tilde{\mathcal{S}}^{\alpha,k_1}_{\tilde{D}} & -\tilde{\mathcal{S}}^{\alpha,k_0}_{\tilde{D}}\\
			-\frac{1}{2}\operatorname{Id}+(\tilde{\mathcal{K}}^{-\alpha,k_1}_{\tilde{D}})^\ast & -\delta \left[\frac{1}{2}\operatorname{Id}+(\tilde{\mathcal{K}}^{-\alpha,k_0}_{\tilde{D}})^\ast\right]
		\end{pmatrix}, \] 
		forms a Dirac cone at $ \alpha_1  $ and $ \alpha_2  $:
		\begin{gather}
			\begin{aligned}
				\omega_{1}(\alpha) &= \omega^{\ast}-\lambda|\alpha-\alpha_j|[1+\mathcal{O}(|\alpha-\alpha_j|)],\\
				\omega_{2}(\alpha) &= \omega^{\ast}+\lambda|\alpha-\alpha_j|[1+\mathcal{O}(|\alpha-\alpha_j|)].
			\end{aligned}
		\end{gather} 
		Here, $ j=1,2 $ and $ \lambda $ is a constant independent of $ \alpha $. And the error term $ \mathcal{O}(|\alpha-\alpha_j|) $ is uniform in $\delta$. As $\delta\to 0$, we have the following asymptotic formula
		\[
		0\neq\omega^\ast = \mathcal{O}(\sqrt{\delta}),\quad 0\neq \lambda = \mathcal{O}(\sqrt{\delta}).
		\]
	\end{theorem}
	Combining this theorem with \Cref{thm:Decomposition} and \Cref{proposi:capacitanceMatEig}, we can prove the following theorem, which guarantees the existence of double Dirac cone near $ \alpha=0 $ for super honeycomb-structured inclusions. 
	\begin{theorem}\label{thm:FourDirac}
		For sufficient small $ |\omega^\ast|>|\alpha|\ge 0 $ and $ \delta $, the second to fifth band functions separate from the first and sixth band, forming a double Dirac cone at $ \alpha = 0 $,
		\begin{gather*}
			\begin{aligned}
				\omega_{j}(\alpha,0,\delta) = \omega^\ast - \lambda|\alpha|(1+\mathcal{O}(|\alpha|)),\\
				\omega_{l}(\alpha,0,\delta) = \omega^\ast + \lambda|\alpha|(1+\mathcal{O}(|\alpha|)).
			\end{aligned}
		\end{gather*}
		where $ j = 2,3 $, $ l= 4,5 $. Here $ \omega^\ast $ and $ \lambda $ are the same as in those in \Cref{thm:DiracConeSmall}. 
	\end{theorem}
	\begin{proof}
	This can be simply proved by noticing 
	\[ \mathcal{A}_{\delta}^{\alpha,\omega}\begin{pmatrix}
		\phi^{\alpha}\\\psi^{\alpha}
	\end{pmatrix} =  \tilde{\mathcal{A}}_{\delta}^{\alpha_1+\alpha,\omega}\begin{pmatrix}
	\phi^{\alpha}_1\\\psi^{\alpha}_1
	\end{pmatrix}+  \tilde{\mathcal{A}}_{\delta}^{\alpha_2+\alpha,\omega}\begin{pmatrix}
	\phi^{\alpha}_2\\\psi^{\alpha}_2
	\end{pmatrix}+  \tilde{\mathcal{A}}_{\delta}^{\alpha,\omega}\begin{pmatrix}
	\phi^{\alpha}_3\\\psi^{\alpha}_3
	\end{pmatrix} .\]
	Here 
	\[ \phi^{\alpha} = \sum_{j=1}^{3}\phi_j^{\alpha},\quad \psi^{\alpha} = \sum_{j=1}^{3}\psi_{j}^{\alpha}, \]
	and are determined from \Cref{proposi:Decomposition}. From \Cref{thm:DiracConeSmall}, one can prove that the degeneracy of band functions $ \{\omega_n(\alpha,0,\delta)\}_{n=1}^{\infty} $ at $ \alpha=0 $ is at least four-fold. We now invoke \Cref{thm:EigValAsympt} and \Cref{proposi:capacitanceMatEig} to finally prove that there are exactly four band functions crossing conically at $ \alpha = 0 $.  
	\end{proof}
	
	\section{Numerical Simulation}\label{sec:NumericalTest}
	In this section, we illustrate our results by direct finite element simulation. We apply the method proposed in \cite{Guo2021,Guo2021a} to solve \eqref{eqn:ProblemFormulation}, since the material parameters $ \rho(x),\kappa(x) $ are piece-wise constant in $ Y $. We suppose the inclusions are six circles whose radius is $ 0.086 $. 
	The material parameters are given by $ \rho_{0} = \kappa_{0} = 1 $, $ \rho_{1} = 1/50 $ and $ \kappa_{1} = 50 $. Thus the contrast parameter $ \delta = 1/50 $.\par 
	First, we illustrate the four-fold degeneracy at $ \alpha=0 $ and the double Dirac cone structure near it. We plot the first six bands of \eqref{eqn:ProblemFormulation} on a line segment connecting $ M_1 = -1/2(k_1+k_2) $ and $ M_2 = 1/2(k_1+k_2) $ in the unit dual cell. Each of its points can be determined by a unique real number $ \mu\in[-1,1] $. The results are shown in \Cref{subfig:ConeSegment}. The conical structure can be seen more clearly in \Cref{subfig:SurfCone}, where we plot the second to fifth band functions $ \{\omega_j(\alpha,0,\delta)\}_{j=2}^5 $ near $ \alpha=0 $, where the double Dirac point resides. These results coincide with \Cref{thm:FourDirac}.\par 
	\begin{figure}[htbp] 
		\centering
		\begin{subfigure}{0.33\textwidth}
			\includegraphics[width=\textwidth]{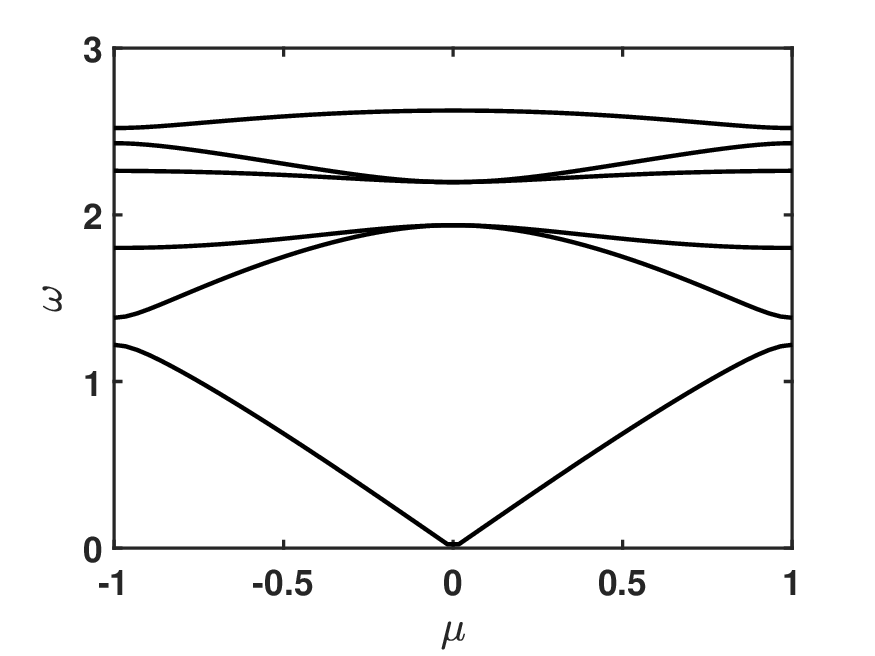}
			\caption{Contract, $ \sigma=-0.1 $.}
			\label{subfig:ContractSegment}
		\end{subfigure}
		\hspace{-0.5cm}
		\begin{subfigure}{0.33\textwidth}
			\includegraphics[width=\textwidth]{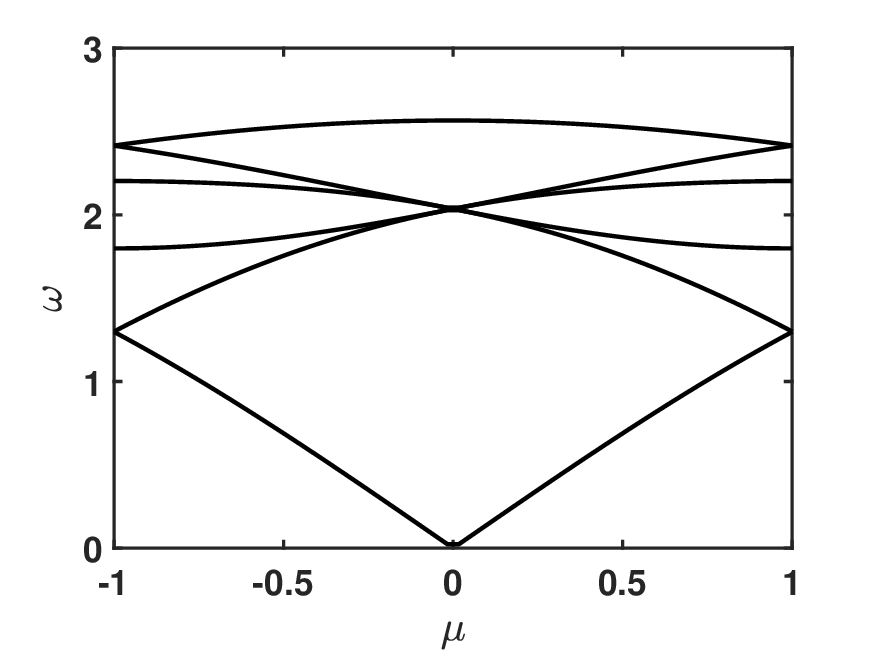}
			\caption{$ \sigma=0 $.}
			\label{subfig:ConeSegment}
		\end{subfigure}
		\hspace{-0.5cm}
		\begin{subfigure}{0.33\textwidth}
			\includegraphics[width=\textwidth]{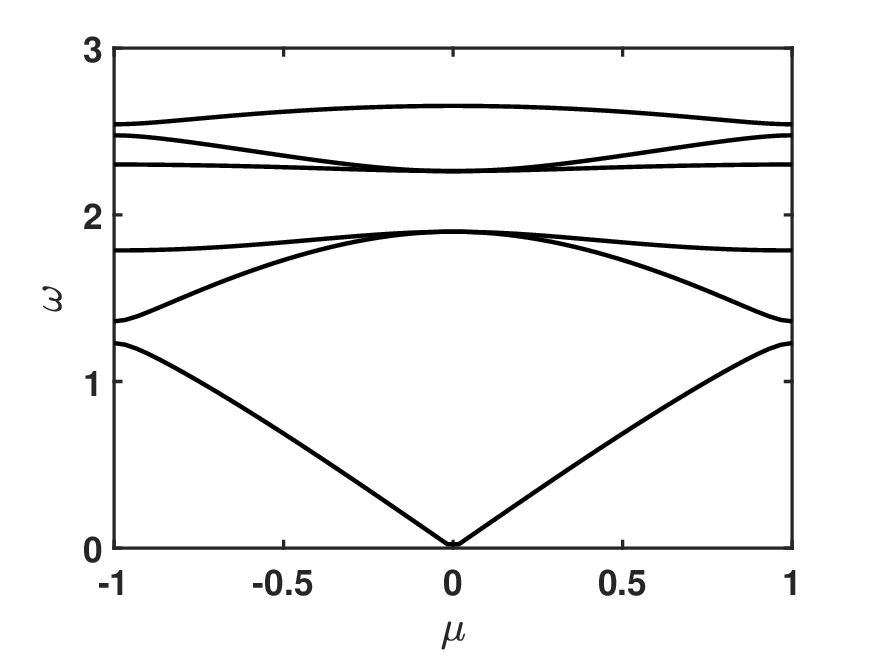}
			\caption{Dilate, $ \sigma=0.1 $.}
			\label{subfig:DilateSegment}
		\end{subfigure}
		\\
		\vspace{-0.4cm}
		\begin{subfigure}{0.33\textwidth}
			\includegraphics[width=\textwidth]{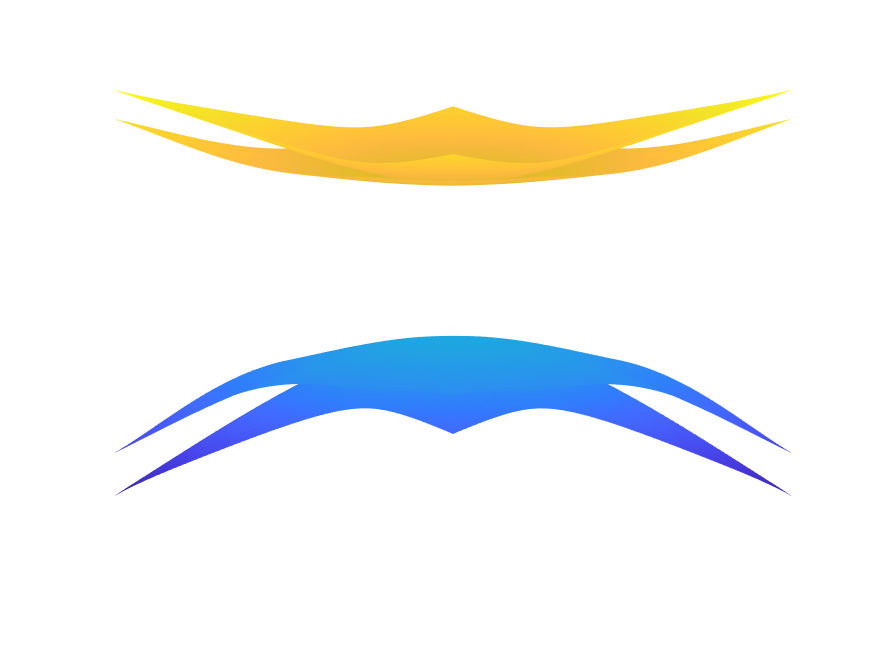}
			\caption{Contract, $ \sigma=-0.1 $.}
			\label{subfig:SurfContract}
		\end{subfigure}
		\hspace{-0.5cm}
		\begin{subfigure}{0.33\textwidth}
			\includegraphics[width=\textwidth]{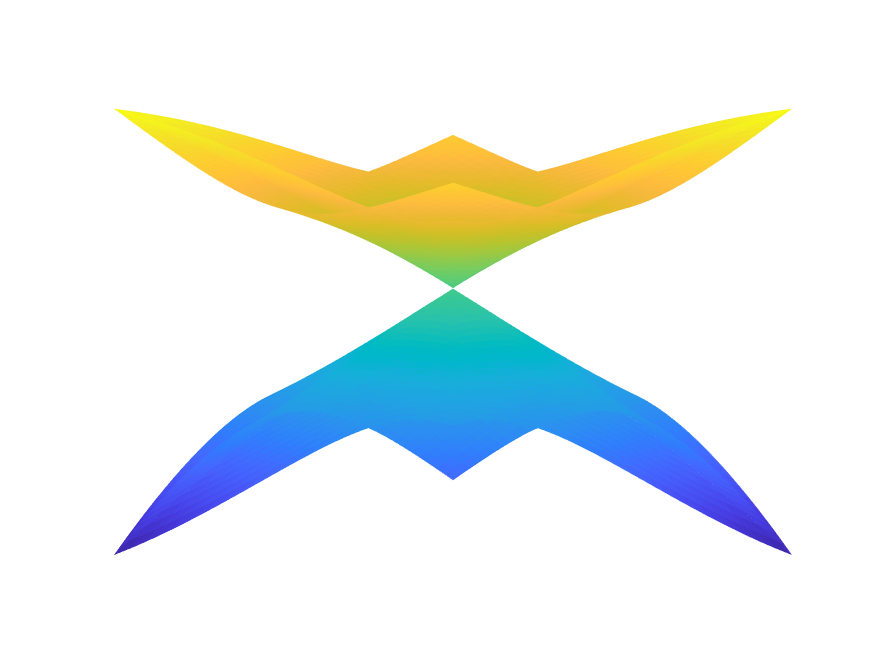}
			\caption{$ \sigma=0 $.}
			\label{subfig:SurfCone}
		\end{subfigure}
		\hspace{-0.5cm}
		\begin{subfigure}{0.33\textwidth}
			\includegraphics[width=\textwidth]{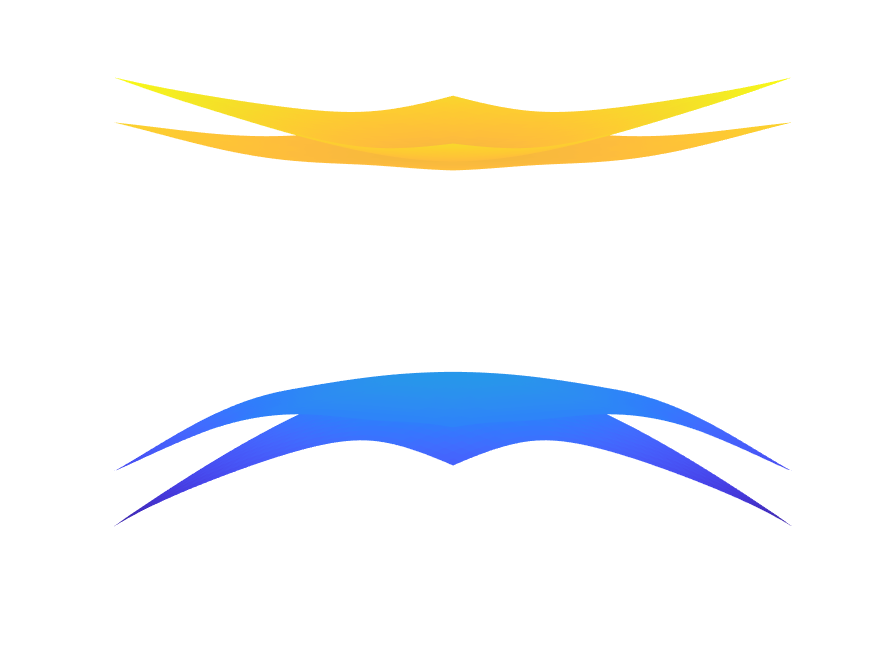}
			\caption{Dilate, $ \sigma=0.1 $.}
			\label{subfig:SurfDilate}
		\end{subfigure}
		\vspace{-0.5cm}
		\caption{\textit{Top Panel: we plot the first six bands $ \{\omega_n(\alpha,\sigma,\delta)\}_{n=1}^6 $ on the line segment $ \overline{M_1M_2} $ in the dual unit cell. Bottom Panel: In these figures we plot the second to fifth dispersion surfaces $ \{\omega_n(\alpha,\sigma,\delta)\}_{n=1}^6$ near $ \alpha=0 $. Here $ \delta = 1/50 $.}}
		\label{fig:BandStructureLine}
	\end{figure}
	Then we plot the figures of how the dilation and contraction affect the band structure $ \{\omega(\alpha,\sigma,\delta)\}_{\alpha\in Y'} $ near $ \alpha=0 $. To be more specific, we let the inclusions be six circles centered at $ \{P_{i,\sigma}\}_{i=1}^6 $. We choose $ \sigma = \mp0.1 $, corresponding to the contracted and dilated cases. Such examples are illustrated in \Cref{fig:LatticePert}. In \Cref{subfig:ContractSegment}, we plot the band structure $ \{\omega^2(\alpha,\sigma,\delta)\}_{\alpha\in Y'} $ when the inclusions are contracted. We can see that the second and third bands detach from the fourth and fifth bands. This can be seen more clearly in \Cref{subfig:SurfContract}, where we plot the second to fifth band functions $ \{\omega^2(\alpha,-0.1,\delta)\}_{\alpha\in Y'} $ near the point $ \alpha=0 $. Similar separation also appears when the inclusions are dilated, see \Cref{subfig:DilateSegment,subfig:SurfDilate}. These results correspond to \Cref{proposi:capacitanceMatEig}.\par 
	Finally, we compute the second to fifth eigenfunction of the original problem \eqref{eqn:ProblemFormulation}. And for simplicity we only discuss the contracted and dilated cases when $\sigma = \mp 0.1$. We first plot these figures in \Cref{fig:EigenContractDilate}. \par 
	\begin{figure}[htbp] 
		\centering
		\begin{subfigure}{0.25\textwidth}
			\includegraphics[width=\textwidth]{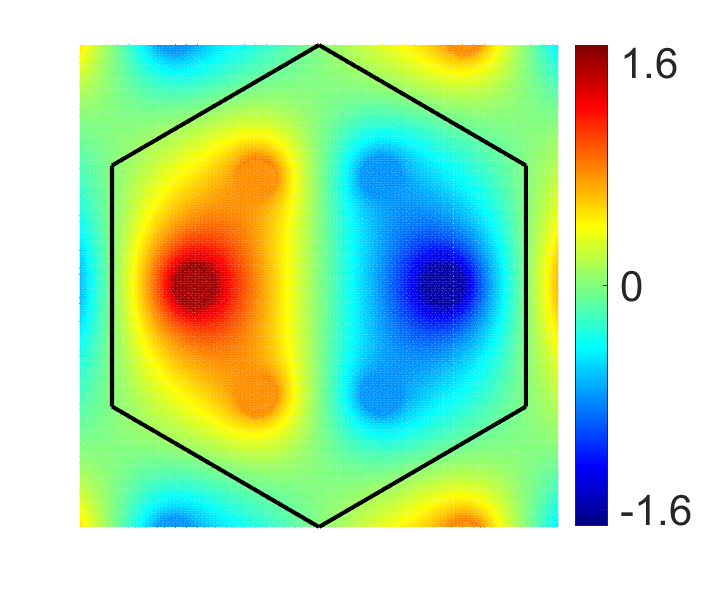}
			\caption{Second, $\sigma=-0.1$}
		\end{subfigure}
		\hspace{-0.3cm}
		\begin{subfigure}{0.25\textwidth}
			\includegraphics[width=\textwidth]{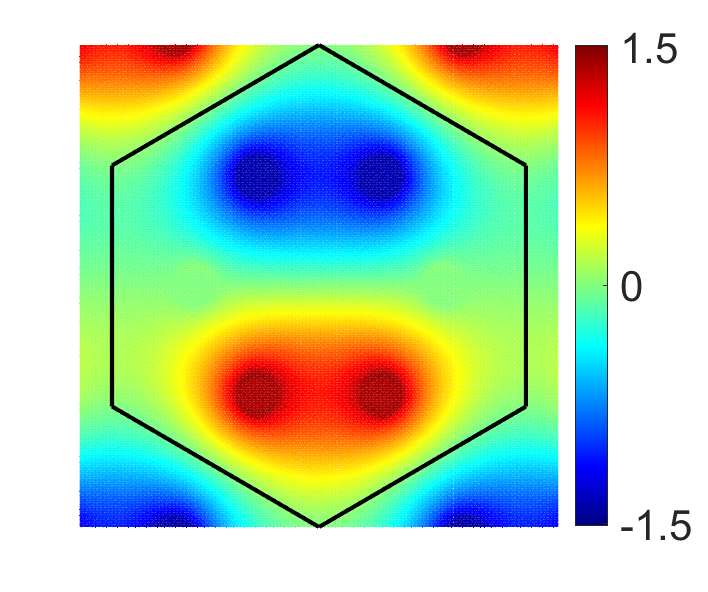}
			\caption{Third, $\sigma=-0.1$}
		\end{subfigure}
		\hspace{-0.3cm}
		\begin{subfigure}{0.25\textwidth}
			\includegraphics[width=\textwidth]{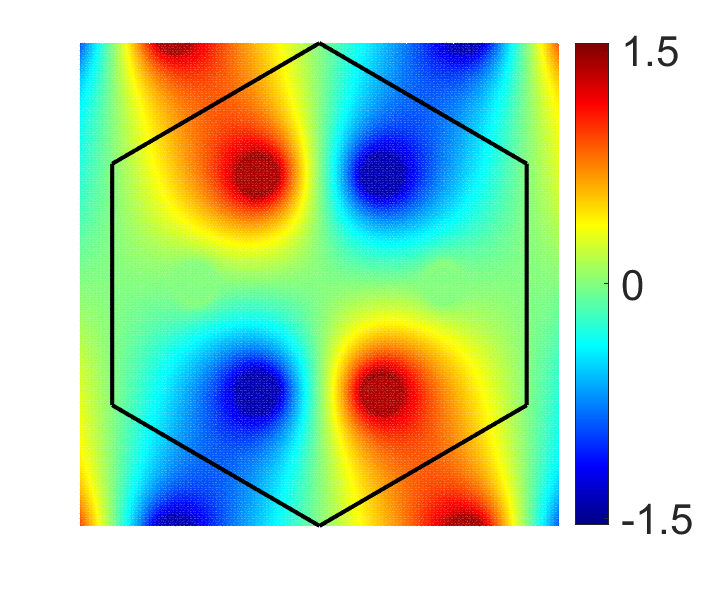}
			\caption{Fourth, $\sigma=-0.1$}
		\end{subfigure}
		\hspace{-0.3cm}
		\begin{subfigure}{0.25\textwidth}
			\includegraphics[width=\textwidth]{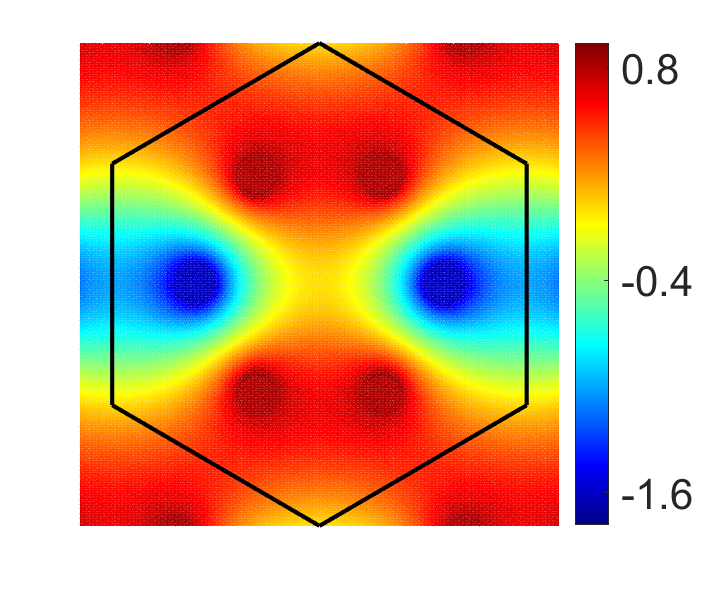}
			\caption{Fifth, $\sigma=-0.1$}
		\end{subfigure}
		\vspace{-0.3cm}\\
		\begin{subfigure}{0.25\textwidth}
			\includegraphics[width=\textwidth]{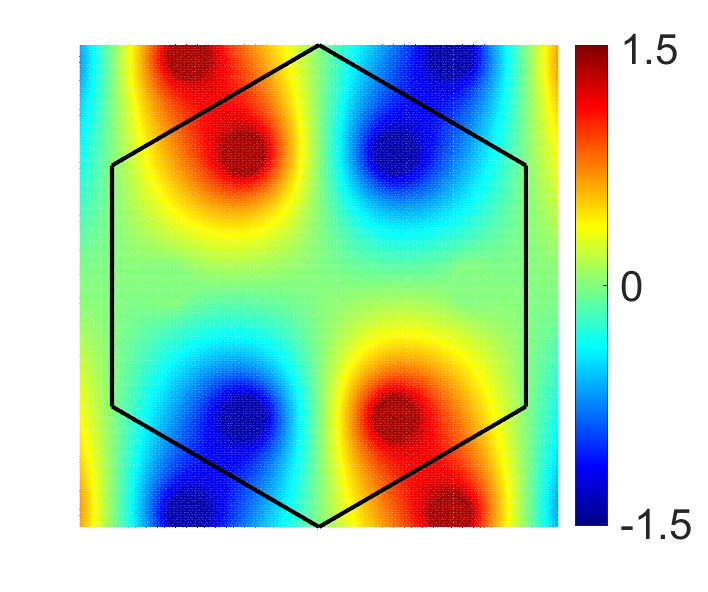}
			\caption{Second, $\sigma=0.1$}
		\end{subfigure}
		\hspace{-0.3cm}
		\begin{subfigure}{0.25\textwidth}
			\includegraphics[width=\textwidth]{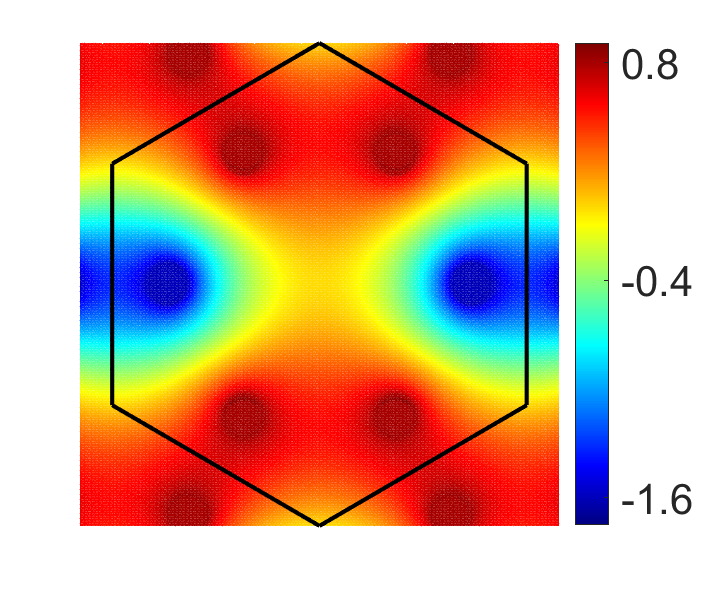}
			\caption{Third, $\sigma=0.1$}
		\end{subfigure}
		\hspace{-0.3cm}
		\begin{subfigure}{0.25\textwidth}
			\includegraphics[width=\textwidth]{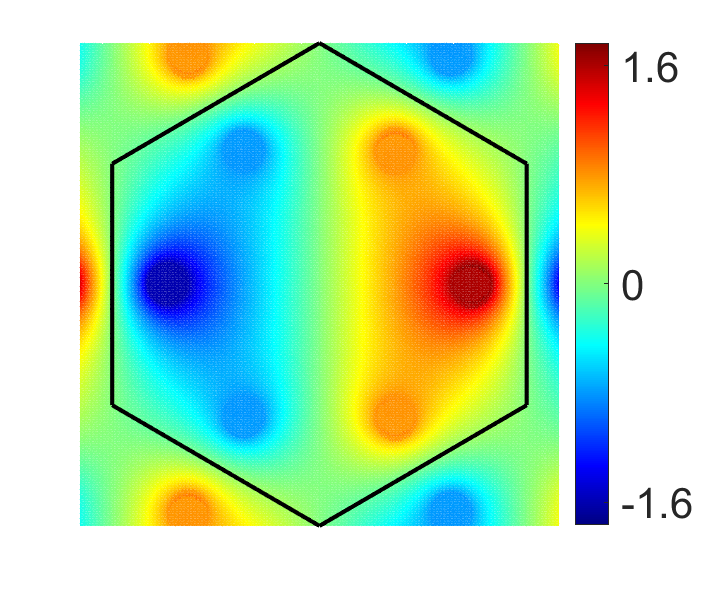}
			\caption{Fourth, $\sigma=0.1$}
		\end{subfigure}
		\hspace{-0.3cm}
		\begin{subfigure}{0.25\textwidth}
			\includegraphics[width=\textwidth]{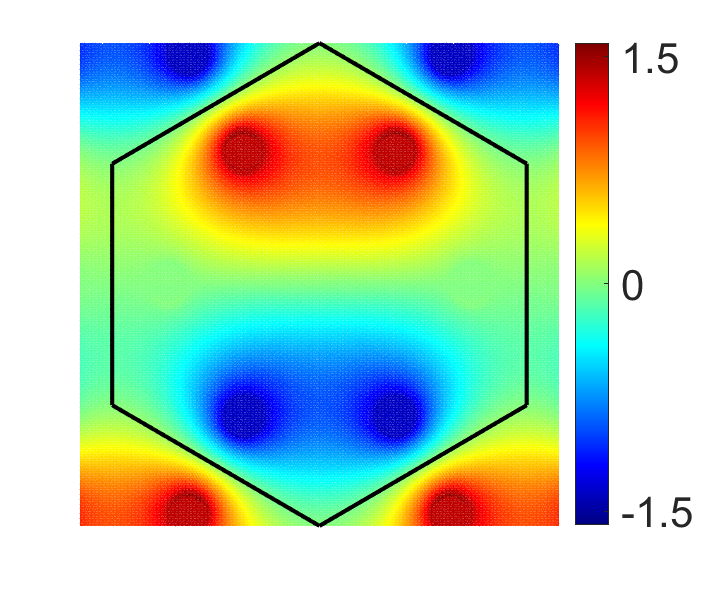}
			\caption{Fifth, $\sigma=0.1$}
		\end{subfigure}
		\vspace{-0.5cm}
		\caption{\textit{In the top panel we plot the second to fifth eigenfunction when the inclusions are contracted. In the bottom panel we plot the second to fifth eigenfunction when the inclosions are dilated. By 'second', 'third', 'fourth', 'fifth' we mean the corresponding eigenfunction. }}
		\label{fig:EigenContractDilate}
	\end{figure}
	We can see from the figure that when the inclusions are contracted, the second and third eigenfunctions are inversion anti-symmetric with respect to the center $O$. When the inclusions are dilated, the second and third eigenfuntions are inversion symmetric with respect to the center $O$. We can conclude that there is a topological phase transition between the contracted and dilated case. This result also coincide with \Cref{proposi:EigFuncPer}.
	
	\appendix
    \section{Periodic Layer Potential and Their Properties}
        Here we define the double layer potential $ \mathcal{D}^{0,0}_D $ for any $ \varphi\in \mathbb{L}^2(\partial D) $:
        \begin{equation}
        	\textstyle\mathcal{D}_D^{0,0}[\varphi](x)\triangleq-\int_{\partial D} n_y\cdot \nabla G^{0,0}(x-y)\varphi(y)d\sigma(y),\quad x\in \mathbb{R}^2\backslash\partial D.
        \end{equation}
    	It satisfies the following jump relations \cite{Ammari2018}
    	\begin{align}
    	\textstyle	 \mathcal{D}^{0,0}_D[\varphi]\big|_\pm(y) &=\textstyle\mp\frac{1}{2}\varphi(y) + \mathcal{K}^{0,0}_D[\varphi](y)\label{apeqn:jumpDoubleValue},\\ 
    		\quad n_y\cdot\mathcal{D}_D^{0,0}[\varphi] \big|_{+}(y)  &\textstyle= n_y  \cdot\nabla \mathcal{D}_D^{0,0}[\varphi] \big|_{-}(y).\label{apeqn:jumpDoubleDeri}
    	\end{align}
    	And we list the mapping properties, whose proofs can be found in \cite{Riva2021}
    	\begin{lemma}[Proposition 12.12, \cite{Riva2021}]\label{aplem:MappingProperty}
    		The following equalities hold
    		\begin{equation}\textstyle
    			\mathcal{K}_D^{0,0}[1] = \frac{1}{2} - \frac{|D|}{|Y|},\quad y \in \partial D,\qquad 
    			\mathcal{D}_D^{0,0}[1] = - \frac{|D|}{|Y|},\quad x \in Y\backslash D.
    		\end{equation}
    	\end{lemma}
    	From this one can prove 
    	\begin{lemma}\label{aplem:KernelKOp}
    	The operator $ -\frac{1}{2}\operatorname{Id} + \mathcal{K}_D^{0,0} $, regarded as an operator from $ \mathbb{L}^{2}(\partial D) $ to itself, has non-trivial kernel of 5 dimension. Further, for every $\psi\in \ker( -\frac{1}{2}\operatorname{Id} + \mathcal{K}_D^{0,0} )$, it satisfies
    	\[ \textstyle\psi = \sum_{j=1}^6c_j\chi_{_{\partial D_j}},\quad \sum_{j=1}^6c_j = 0. \]
    	\end{lemma}
    	\begin{proof}
    		For any $ \psi $ that satisfies $ (-\frac{1}{2}\operatorname{Id} + \mathcal{K}_D^{0,0})[\psi] = 0 $, we have, by the jump relations \eqref{apeqn:jumpDoubleValue}, 
    		$ \left.\mathcal{D}_D^{0,0}[\psi]\right|_{+} = 0 $ and $ \mathcal{D}_D^{0,0}[\psi] $ is harmonic in $ Y\backslash D $\cite{Riva2021}. It immediately follows that, by \Cref{thm:Uniqueness}, $ \mathcal{D}_D^{0,0}[\psi] \equiv 0 $ for all $ x\in Y\backslash D $. \par 
    		By the jump relation \eqref{apeqn:jumpDoubleDeri} and the existence result, $ \mathcal{D}_D^{0,0}[\psi] \equiv c_{j} $ on each inclusion $ D_{j} $. And it can be verified that $\psi = c_{j}$ on $\partial D_{j}$ and for $ x\in Y\backslash D $
    		\begin{equation*}
    			\textstyle\mathcal{D}_D^{0,0}[\psi] = \mathcal{D}_D^{0,0}[\sum_{j=1}^{6}c_{j}\chi_{_{D_j}}] = -\sum_{j=1}^{6}\frac{|D|}{6|Y|}c_{j} = -\frac{|D|}{6|Y|}\sum_{j=1}^{6}c_{j}=0.
    		\end{equation*}   
    		So $ (c_{j})_{j=1}^{6} $, if regarded as a vector in $ \mathbb{R}^{6} $, must satisfy $ \sum_{j=1}^{6}c_{j}=0 $. Since all such vectors form a linear space of dimension 5, it means $ \operatorname{dim} \ker( -\frac{1}{2}\operatorname{Id} + \mathcal{K}_D^{0,0} ) =5 $.
    	\end{proof}
    	Now we can state and prove the main proposition concerning  $\ker[ -\frac{1}{2}\operatorname{Id} + (\mathcal{K}_D^{0,0})^\ast] $:
    	\begin{proposi}\label{approposi:KernelKastOp}
    		The kernel of the operator $ -\frac{1}{2}\operatorname{Id}+(\mathcal{K}^{0,0}_D)^\ast $ is five dimensional. Moreover, for any $\varphi\in \ker[-\frac{1}{2}\operatorname{Id}+(\mathcal{K}^{0,0}_D)^\ast]$,
    		\begin{equation}\label{apeqn:SumZero}
    		\textstyle \mathcal{S}_D^{0,0}[\varphi] = \sum_{i=1}^{6}a_i\chi_{_{\partial D_i}}, 
    		\end{equation}
    		where $ \{ a_{i} \}_{i=1}^6 $ are real constants satisfying
    		$\sum_{i=1}^6a_i = 0. $
    	\end{proposi}
    	\begin{proof}
    		The first part is an immediate corollary from \Cref{aplem:KernelKOp}, and noticing that 
    		\[ \textstyle \operatorname{dim} \ker\big[ -\frac{1}{2}\operatorname{Id} + (\mathcal{K}_D^{0,0})^\ast \big]  = \operatorname{dim} \ker\big[ -\frac{1}{2}\operatorname{Id} + \mathcal{K}_D^{0,0} \big] . \]\par 
    		Now suppose $\varphi_j \in \ker\big[ -\frac{1}{2}\operatorname{Id} + (\mathcal{K}_D^{0,0})^{\ast} \big] $ for $ j = 1,2,\ldots,5 $. We have, by definition and \Cref{aplem:MappingProperty}:
    		\begin{equation*}\textstyle
    			\frac{1}{2}\int_{\partial D} \varphi_j d\sigma(y) = \int_{\partial D} (\mathcal{K}^{0,0}_D)^{\ast}[\varphi_j]d\sigma(y)= \int_{\partial D} \varphi_j \mathcal{K}_D^{0,0}[1]d\sigma(y)= (\frac{1}{2} - \frac{|D|}{|Y|})\int_{\partial D} \varphi_j d\sigma(y).
    		\end{equation*}
    		This implies $ \varphi_j\in \mathbb{L}^2_{0}(\partial D) $. Thus, $\mathcal{S}_D^{0,0}[\varphi_j]$ is harmonic in $ D $ and has homogeneous Neumann boundary on $ \partial D $. By \Cref{thm:Uniqueness}, $ \mathcal{S}_D^{0,0}[\varphi_j] \equiv d_{jk} $ in $ D_{k} $. 
    		If for all $ j $, \eqref{apeqn:SumZero} holds, then we are done. Otherwise, suppose we have $ \sum_{l=1}^{6}(d_{jl}-c_j)=0,\quad j=1,2,\ldots,5. $
    		Since we can regard $ ( d_{jl})_{l=1}^{6} $ as vectors in $ \mathbb{R}^6 $. Obviously the vectors $ \{ (d_{jl}) \}_{j=1}^{5} $ are linearly independent.  
    		And we argue by contradictory that the vectors $ \{(d_{jl}-c_j)\}_{j=1}^5 $ are linearly independent for any real numbers $ \{c_j\}_{j=1}^{5} $. \par 
    		Suppose otherwise there exist $ \{a_j\}_{j=1}^{5} $ that does not equals to zero at the same time, such that the following equality holds for some $ \{c_j\}_{j=1}^{5} \subset \mathbb{R}$:
    		\[ \textstyle\sum_{j=1}^{5}a_{j}(d_{jl})_{l=1}^{6} = \sum_{j=1}^{5}a_jc_j.\]
    		However, the above equality is equivalent to  
    		\[ \textstyle \sum_{j=1}^5\mathcal{S}^{0,0}_{D}[a_j\varphi_j](y) = \sum_{j=1}^5a_jc_j,\quad y\in \partial D. \]
    		By \Cref{lem:uniqueness}, the sum $ \sum_{j=1}^5a_j\varphi_j \equiv 0 $, which is a contradictory. \par 
    		Now we can assume, after a certain invertible linear transformation, there exist $ \{\psi_{j}\}_{j=1}^{5}\subset \ker[-\frac{1}{2}\operatorname{Id} + (\mathcal{K}^{0,0})^{\ast}] $ and $ \{ d_{j} \}_{j=1}^{5}\subset \mathbb{R} $, such that 
    		\begin{equation}\label{apeqn:SingleSum}
    			[\mathcal{S}_D^{0,0}[\psi_{j}] + d_{j}]|_{D_{k}}= \eu^{\iu jk\pi/3},\quad j = 1,2,\ldots 5.
    		\end{equation}
    		We have, by definition 
    		\[ [\mathcal{S}^{0,0}_{D}[\psi_j](Ry)+d_j ]|_{D_{k}}= \eu^{\iu j(k+1)\pi/3},\quad  j = 1,2,\ldots,5. \]
    		Multiplying the constant $ \eu^{\iu j\pi/3} $ on both sides of \eqref{apeqn:SingleSum}, we have 
    		\[[ \mathcal{S}^{0,0}_{D}[\eu^{\iu j\pi/3}\psi_j](y) + d_j\eu^{\iu j\pi/3}]|_{D_{k}} = \eu^{\iu j(k+1) \pi/3},\quad  j = 1,2,\ldots,5.  \] 
    		Combining the above two equations, we have
    		\begin{equation}
    			\mathcal{S}^{0,0}_{D}[\psi_j(Ry) - \eu^{\iu j\pi/3}\psi_j(y)]|_{D_k} = d_j(\eu^{\iu j\pi/3}-1),\quad j = 1,2,\ldots,5.
    		\end{equation}
    		By \Cref{lem:uniqueness}, we have $ d_j = 0 $, and we have finished the proof of \Cref{approposi:KernelKastOp}. 
    		\end{proof} 
   		\section{Representation of Solution}\label{apsec:RepresentSol}
   		In this section we justify that the non-constant solution to the original problem \eqref{eqn:ProblemFormulation} when $ \omega=0 $ can be represented by single layer potentials. From definition we have 
   		\begin{equation}\label{apeqn:PeriodicGreen}
   			\textstyle \Delta G^{0,0} = \sum_{n\in \Lambda} \delta(x-n)- \frac{1}{|Y|}.
   		\end{equation}
   		So when $ x\in D $ we have by \eqref{apeqn:PeriodicGreen}
   		\begin{subequations}
   			\begin{align*}
   				\textstyle u(x)-\frac{1}{|Y|}\int_{D} & \textstyle u(z)dz=\int_{D} u(z)\Delta G^{0,0}(x-z)-\Delta u(z) G^{0,0}(x-z)dz \\&\textstyle= 
   				\int_{\partial D}n_y\cdot \nabla G(x-y)u|_{-}(y) -\left.n_{y}\cdot \nabla u(y)\right|_{-}G^{0,0}(x-y)d\sigma(y).
   			\end{align*}
   		\end{subequations}	
   		And we can obtain the following equality similarly 
   		\begin{subequations}
   			\begin{align*}
   				\textstyle-\frac{1}{|Y|}\int_{Y\backslash \overline{D}} u(z)dz&=\textstyle\int_{Y\backslash \overline{D}} u(z)\Delta G^{0,0}(x-z)-\Delta u(z) G^{0,0}(x-z)dz \\&\textstyle= 
   				\int_{\partial D}-n_y\cdot \nabla G(x-y)u|_{+}(y) +\left.n_{y}\cdot \nabla u(y)\right|_{+}G^{0,0}(x-y)d\sigma(y).
   			\end{align*}
   		\end{subequations}
   			Combining the above two equalities and the jump condition in \eqref{eqn:ProblemFormulation}, we have
   		\begin{equation}
   			\textstyle u(x)-\frac{1}{|Y|}\int_{Y} u(z)dz=\mathcal{S}^{0,0}_{D}[ \left.n_{y}\cdot \nabla u(y)\right|_{+} - \left.n_{y}\cdot \nabla u(y)\right|_{-} ](x).
   		\end{equation}
   		Hence the constant solution cannot be represented by single layer potentials. 
   		\section{Uniqueness of Laplace Equation}
   		To make this article more self-contained, we list some uniqueness results of the Laplace equation in $ D  $ or $ Y\backslash D $ with Dirichlet boundary condition.
   		\begin{theorem}[Theorem 12.16 and Theorem 12.17 in \cite{Riva2021}]\label{thm:Uniqueness}
   			Let $ u\in C^{1,\alpha}(D) $ and $ v\in C^{1,\alpha}(Y\backslash D) $ be $ \Lambda $-periodic functions that satisfy the following interior Dirichlet and exterior Dirichlet problems:
   			\begin{gather*}
   			\textstyle\left\{\begin{aligned}
   					&\textstyle\Delta u = 0, \quad x\in D,\\
   					&\textstyle u=0,\quad y \in \partial D,\\
   					&\textstyle u(x+l) = u(x),\quad \forall x\in D, l\in \Lambda.
   			\end{aligned}\right. \qquad 
   			\left\{ \begin{aligned}
   				&\textstyle\Delta v = 0, \quad x\in Y\backslash D,\\
   				&\textstyle v=0,\quad y \in \partial D,\\
   				&\textstyle v(x+l) = v(x),\quad \forall x\in Y\backslash D, l\in \Lambda.
   			\end{aligned} \right.
   			\end{gather*}
   			Then $ u \equiv 0 $ and $ v\equiv 0 $ in $ D $ and $ Y\backslash D $ correspondingly.
   		\end{theorem}
   		\section*{Acknowledgments}
	We would like to acknowledge the assistance of Prof. Habib Ammari for insightful discussions. We would also like to acknowledge the assistance of Ms. Ying Cao for interesting discussions.
	\bibliographystyle{siamplain}
	\bibliography{references}
\end{document}

%% file: ex_shared.tex
\usepackage{lipsum}
\usepackage{amsfonts}
\usepackage{graphicx}
\usepackage{subcaption}
\usepackage{epstopdf}
\usepackage{algorithmic}
\ifpdf
  \DeclareGraphicsExtensions{.eps,.pdf,.png,.jpg}
\else
  \DeclareGraphicsExtensions{.eps}
\fi

\newsiamremark{remark}{Remark}
\newsiamremark{hypothesis}{Hypothesis}
\newsiamremark{assumption}{Assumption}
\crefname{hypothesis}{Hypothesis}{Hypotheses}
\newsiamthm{claim}{Claim}
\newsiamthm{proposi}{Proposition}
\newsiamremark{defini}{Definition}
\def\eu{\ensuremath{\mathrm{e}}}
\def\iu{\ensuremath{\mathrm{i}}}

\headers{Generalized Honeycomb}{Borui Miao and Yi Zhu}

\title{Generalized Honeycomb-structured Materials in the Subwavelength Regime\thanks{Submitted to the editors DATE.
\funding{This work was funded by the National Key R\&D Program of China (Grant No. 2021YFA0719200) and the National Natural Science Foundation of China (Grant No. 11871299).}}}

\author{Borui Miao\thanks{Yau Mathematical Sciences Center, Tsinghua University, Beijing, 100084, China (\email{mbr20@mails.tsinghua.edu.cn}).}
\and Yi Zhu\thanks{Yau Mathematical Sciences Center, Tsinghua University, Beijing, 100084, China and Yanqi Lake Beijing Institute of Mathematical Sciences and Applications, Beijing, 101408, China (\email{yizhu@tsinghua.edu.cn}).}
}

\usepackage{amsopn}

\makeatletter
\newcommand*{\addFileDependency}[1]{
  \typeout{(#1)}
  \@addtofilelist{#1}
  \IfFileExists{#1}{}{\typeout{No file #1.}}
}
\makeatother
